\documentclass[final,onefignum,onetabnum]{siamart220329}
\usepackage{amsfonts}
\usepackage{graphicx}
\usepackage{epstopdf}
\usepackage{caption,subcaption} 
\usepackage{xcolor} 
\usepackage{booktabs} 
\usepackage{hyperref} 
\ifpdf
  \DeclareGraphicsExtensions{.eps,.pdf,.png,.jpg}
\else
  \DeclareGraphicsExtensions{.eps}
\fi

\newsiamremark{remark}{Remark}
\newsiamremark{example}{Example}
\crefname{hypothesis}{Hypothesis}{Hypotheses}
\newsiamthm{claim}{Claim}

\headers{Second-order flows via convex-splitting schemes}{H. Chen, G. Dong, J. A. Iglesias, W. Liu, and Z. Xie}

\title{Second-order flows for approaching stationary points of a class of non-convex energies via convex-splitting schemes\thanks{Submitted to the editors DATE.
\funding{The work of HC and GD was supported by the NSFC grant 12001194. GD is also partially supported by the NSFC grant 12471402 and the NSF grant of Hunan Province 2024JJ5413. The work of WL was supported by the NSFC grant 12101252 and NUDT grant 202402-YJRC-XX-002. The work of HC and ZX was supported by the Key Project of Xiangjiang Laboratory 22XJ01013 and the NSFC grant 12171148.}}}

\author{Haifan Chen\thanks{MOE-LCSM, School of Mathematics and Statistics, Hunan Normal University, Changsha 410081, China 
  (\email{haifanchen@hunnu.edu.cn}, \email{ziqingxie@hunnu.edu.cn}).}
\and Guozhi Dong\thanks{School of Mathematics and Statistics, HNP-LAMA, Central South University, Changsha 410083, China 
  (\email{guozhi.dong@csu.edu.cn}).}
\and Jos\'e A. Iglesias\thanks{Department of Applied Mathematics, University of Twente, 7500 AE Enschede, The Netherlands 
  (\email{jose.iglesias@utwente.nl}).}
\and Wei Liu\thanks{Department of Mathematics, National University of Defense Technology, Changsha 410073, China 
  (\email{wl@nudt.edu.cn}).}
\and Ziqing Xie\footnotemark[2]}

\usepackage{amsopn}

\ifpdf
\hypersetup{
  pdftitle={Second-order flows via convex-splitting},
  pdfauthor={H. Chen, G. Dong, J. A. Iglesias, W. Liu, and Z. Xie}
}
\fi

\begin{document}

\maketitle

\begin{abstract}
This paper contributes to the exploration of a recently introduced computational paradigm known as second-order flows, which are characterized by novel dissipative hyperbolic partial differential equations extending accelerated gradient flows to energy functionals defined on Sobolev spaces, and exhibiting significant performance particularly for the minimization of non-convex energies. Our approach hinges upon convex-splitting schemes, a tool which is not only pivotal for clarifying the well-posedness of second-order flows, but also yields a versatile array of robust numerical schemes through temporal (and spatial) discretization. We prove the convergence to stationary points of such schemes in the semi-discrete setting. Further, we establish their convergence to time-continuous solutions as the timestep tends to zero. Finally, these algorithms undergo thorough testing and validation in approaching stationary points of representative non-convex variational models in scientific computing.
\end{abstract}

\begin{keywords}
non-convex functional,
second-order flow, 
dissipative hyperbolic equation, 
convex-splitting scheme,
stationary point,
Ginzburg-Landau energy, 
Landau-de Gennes model
\end{keywords}

\begin{MSCcodes}
35Q90, 35L71, 65M15, 65K10 
\end{MSCcodes}

\section{Introduction}
Second-order flow is a concept proposed in \cite{CheDonLiuXie23}, referring to certain dissipative second-order hyperbolic partial differential equations (PDEs), which serves as a compelling alternative to gradient flows for approaching minimizers of variational problems. (Non-convex) variational models and PDEs are central topics in scientific computing and applied mathematics, offering useful tools to describe a variety of complex phenomena across multiple disciplines. The current paper is motivated to develop robust and efficient numerical methods based on second-order flows for solutions of a class of non-convex variational problems, e.g., the Ginzburg-Landau energy in phase-field modelings and the Landau-de Gennes energy of the Q-tensor model for liquid crystals.

These variational models are grounded in a common mathematical framework. Let \( H \) be a Hilbert space of the vector-valued function \( \mathbf{u}:\Omega\subset \mathbb{R}^d\rightarrow\mathbb{R}^N (d=1,2,3)\), where \( \Omega\) is a bounded domain. The goal is to find a function \( \mathbf{u}_g \in H \) that minimizes an energy functional \( E \), formulated as
\begin{align}\label{eq:energy_minimization}
E_g := E(\mathbf{u}_g) = \min_{\mathbf{u} \in H} E(\mathbf{u}),
\end{align}
with the energy functional given by
\begin{align}\label{eq:energy}
E(\mathbf{u}) = \int_{\Omega} \left( \frac{1}{2} |\nabla \mathbf{u}|^2 + F(\mathbf{u}) \right) \, \mathrm{d} \mathbf{x},
\end{align}
where \( \nabla\) denotes the gradient operator, and \(|\nabla\mathbf{u}|^2 = \sum_{i=1}^{N} \sum_{j=1}^{d} \left( \frac{\partial u_i}{\partial x_j} \right)^2  \). This paper focuses on a typical nonlinear potential $F$ of the form 
\[ F(\mathbf{u})=\frac{\alpha}{4}|\mathbf{u}|^4-\frac{\beta}{2}|\mathbf{u}|^2+\gamma \] 
with constants \( \alpha, \beta, \gamma \) all positive. The associated Euler-Lagrange equation for this variational problem is
\begin{align}\label{eq:euler-lagrange}
-\Delta \mathbf{u} + f(\mathbf{u}) = 0 \quad \text{in }\Omega,
\end{align}
where \( \Delta \mathbf{u} \) represents the Laplacian of \( \mathbf{u} \), and \( f(\mathbf{u})= \alpha |\mathbf{u}|^2\mathbf{u} - \beta \mathbf{u} \) is derived from the variation of the nonlinear potential term \( F(\mathbf{u}) \). 

Gradient flow methods have been quite popular to tackle this type of problems \cite{DuLiuWan04,DuLiuWan06,IyeXuZar15,CaiSheXu17,LiuQiaZha21,LiuQiaZha23}, owing to their obvious advantages: (i) only the gradient information (or first-order variational derivative) of the objective functional $E(\mathbf{u})$ is required, which enables easy implementation through various discretization strategies;  (ii) the energy stability is often preserved due to the dissipative mechanism. Notably, employing the $L^2$-gradient flow of the energy \eqref{eq:energy} leads to the derivation of a parabolic equation, which can be expressed as:
\begin{align}
\label{eq:parabolic_equation}
    \partial_t \mathbf{u} = \Delta \mathbf{u} - f(\mathbf{u}), \quad \text{for } \mathbf{x} \in \Omega, \, t > 0,
\end{align}
subject to suitable boundary conditions on \(\partial\Omega\times\{t\,|\,t\geq0\}\). This equation is reminiscent of the well-known Allen-Cahn equation (system), a representative equation of phase-field type models.

The idea of second-order flows is rooted in a recent topic in convex optimization. Considering some cost function $J:H\to \mathbb{R}$ for minimization, second-order dynamics of the following type have been of high interest in the literature \cite{EdvSveGulPer12,SuBoyCan16,AttChbPey18,AttBotCse22}
\begin{align}
\label{eq:second-order flow}
\ddot{\mathbf{u}} + \eta(t) \dot{\mathbf{u}} &= -\nabla J(\mathbf{u}),
\end{align}
where \(\eta(t) > 0\) is a damping coefficient.
When $J(\mathbf{u})$ is a convex function over finite dimensional spaces, i.e., \eqref{eq:second-order flow} is a system of ordinary differential equations (ODEs), such a formulation is known to offer distinct advantages referring to the recent progress of second-order inertial dynamics in optimization. This was mainly sparked by the work of Su et al. \cite{SuBoyCan16}, which established a connection between the second-order ODEs and Nesterov’s accelerated gradient method \cite{Nes83}. Subsequent research, including studies \cite{AttChbPey18,ZhaBer20,OG2020EJP,BotDonElb22,AttBotCse22}, has continued to explore the theoretical and practical aspects of these second-order ODEs. In this work we focus on \emph{second-order flows particularly corresponding to PDEs}. The study of such second-order flows faces additional complexities, particularly in their theoretical understanding and numerical analysis, due to the fact that $\nabla J(\mathbf{u})$ involves spatial differential operators, see, e.g., \cite{DonHinZha21}. In a recent work \cite{CheDonLiuXie23}, the authors have introduced two distinct types of second-order flows as strategies for minimizing the Gross-Pitaevskii functional under the non-convex constraints, which is a fundamental model for simulating the ground states of rotating Bose-Einstein condensates (BECs). The newly introduced second-order flow methods, incorporating both explicit and semi-implicit temporal discretizations, have demonstrated notable improvements over gradient flow type approaches in terms of computational efficiency.

Investigating the analytical aspects, numerical methods and application of second-order flows for (non-convex) variational problems in Sobolev spaces offers an intriguing and largely unexploited area of research. Despite some of the initial treatment as explored in \cite{DonHinZha21}, and progress on numerical efficiency made in the previous work \cite{CheDonLiuXie23}, a comprehensive theoretical and numerical understanding of second-order flows, especially in the context of non-convex variational problems, remains a challenging open topic. That motivates the current work, where we try to go one step further towards non-convex variational problems with second-order flows. In particular, we wish to establish some foundations, e.g. their well-posedness and convergence to stationary points. A full convergence analysis is particularly challenging with non-convex energies, where even in the finite-dimensional setting the use of some form of \L{}ojasiewicz-type inequalities is required \cite{BegBolJen15, JiaSiCheBao20}. In the PDE setting, the adequate notion is that of \L{}ojasiewicz-Simon inequalities, which were first introduced in \cite{Sim83} and have been applied in various contexts \cite{BaoCheJiaQiu24}. However, for their application to semilinear second-order dissipative equations as in \cite{HarJen99} and related works, the specific energies \eqref{eq:energy} that we consider here, the vectorial setting and strength of the nonlinearities prevent such methods from being directly applicable.

When taking $H=H_0^1(\Omega,\mathbb{R}^N)$ and formulating a second-order flow to approach a minimizer or stationary point of \( E(\mathbf{u}) \) in \eqref{eq:energy}, we are led to the following dissipative hyperbolic PDE:
\begin{align}
\label{eq:damped_hyperbolic_equation}
    \partial_{tt} \mathbf{u} + \eta(t) \partial_t \mathbf{u} &= \Delta \mathbf{u} - f(\mathbf{u}), \quad \mathbf{x}\in \Omega, \, t > 0,
\end{align}
with initial data \( \mathbf{u}(\cdot, 0) = \mathbf{u}^0 \) and \( \partial_t \mathbf{u}(\cdot, 0) = \mathbf{v}^0 \) in \( \Omega \). For simplicity, we consider  homogeneous Dirichlet boundary conditions on $\partial\Omega\times\{t\,|\,t\geq0\}$. Denoting  \(\mu\) to be the first-order variational derivative of $E$ (or chemical potential), 
$$
\mu:=-\Delta \mathbf{u} + f(\mathbf{u}) \in H^{-1}(\Omega, \mathbb{R}^N),
$$
the damped hyperbolic PDE \eqref{eq:damped_hyperbolic_equation} can be reformulated as the system
\begin{align}
\partial_t\mathbf{v}& =-\mu-\eta(t)\mathbf{v}, \label{eq:rewritten_damped_hyperbolic1}\\ 
\partial_t \mathbf{u} & = \mathbf{v}. \label{eq:rewritten_damped_hyperbolic2}
\end{align}
Building on insights from \cite{CheDonLiuXie23}, we note a fundamental distinction between the gradient flow \eqref{eq:parabolic_equation}, which exhibits energy-decaying of \( E \), and the second-order flow \eqref{eq:damped_hyperbolic_equation} that dissipates the following pseudo-energy via Lyapunov analysis \cite{AttChbPey18}:
\begin{align}
\label{pseudo-energy-ac}
\mathcal{E}(\mathbf{u}, \mathbf{v}) = E(\mathbf{u}) + \frac{1}{2}\|\mathbf{v}\|_{L^2}^2.
\end{align}

In \cite{CheDonLiuXie23}, a semi-implicit scheme with a stabilization factor was applied to discretizing second-order flows, achieving remarkable computational efficiency. However, this approach lacks rigorous theoretical support, such as unconditional energy stability and convergence analysis. The goal of this paper is to develop robust numerical schemes with theoretical guarantees for second-order flows. The main challenges in the theoretical analysis of these nonlinear hyperbolic PDEs are due to the non-convexity of the energy functional $E:H\to \mathbb{R}$. The convex-splitting method is a proper tool in this regard, renowned for its ability to ensure energy stability and unique solvability, independent of the temporal and spatial step sizes. This method was popularized by Eyre \cite{Eyr98} and widely employed in various contexts, such as phase-field models \cite{FenWis12,DieFenWis15}, thin film epitaxy models \cite{SheWanWanWis12}, phase-field crystal models \cite{HuWisWan09,WisWanLow09}, and modified phase-field crystal models \cite{WanWis11,BasHuLow13,BasLowWan13}. 
Inspired by these successful examples in the literature, we adopt the convex-splitting approach in our study to the hyperbolic PDEs described in \eqref{eq:damped_hyperbolic_equation}.

During the preparation of the paper, we are aware of the works on viscous Cahn-Hilliard equation \cite{Gal01,GalJou05,YanZhaHe18} and the MPFC model \cite{SteHaaPro06,ProDanAth07,WanWis11,BasHuLow13,BasLowWan13}. We would distinguish our work with the existing ones in the literature, as the latter are tailored for simulating specific physical processes. 
Our research employs second-order flows as artificial dynamics, targeting the minimization of the energy functionals in \eqref{eq:energy}, especially their novel numerical algorithms and analysis. Moreover, while the existing models typically incorporate higher-order spatial derivatives, exhibiting higher spatial regularity in the solutions, our investigation into hyperbolic Allen-Cahn type equations, as represented by \eqref{eq:damped_hyperbolic_equation}, confronts the absence of such regularity, adding certain challenges to the analysis. Thus, the distinction in both the mathematical focus and motivation underscores the contribution and novel perspective of our work.

The rest of the paper is structured as follows. Section \ref{section:convergence} begins with a first-order time-discrete, space-continuous convex-splitting scheme, proving its strict pseudo-energy decay and convergence (up to subsequences) to a stationary point of \(E\). By establishing timestep-independent estimates, Section \ref{sec:well-posedness} demonstrates that sequences constructed from numerical solutions of this scheme converge to solutions of the original PDE in \eqref{eq:damped_hyperbolic_equation}. Leveraging this convergence, we subsequently establish the unique existence of a global smooth solution for \eqref{eq:damped_hyperbolic_equation}. We then provide corresponding results for a second-order time-discrete, space-continuous convex-splitting scheme in Section \ref{sec:second-order}. Section \ref{section:numerics} introduces two fully discrete schemes of the second-order flow to provide novel algorithms for the energy minimization problem and presents numerical results to validate the efficiency of our proposed algorithms. To make comparisons, we list other involved relevant algorithms in the Appendix.

We use \((\cdot, \cdot)\) to denote the standard \(L^2\)-inner product for all \(\mathbf{u}, \mathbf{v} \in L^2(\Omega,\mathbb{R}^N)\) as
$$(\mathbf{u}, \mathbf{v} ):=\int_{\Omega} \mathbf{u}(\mathbf{x}) \cdot \mathbf{v}(\mathbf{x})  d \mathbf{x}.$$
Furthermore, \(a(\mathbf{u},\mathbf{v})\) :=\((\nabla\mathbf{u},\nabla\mathbf{v})=\int_{\Omega} \nabla\mathbf{u}(\mathbf{x}):\nabla\mathbf{v}(\mathbf{x}) d \mathbf{x}\). For brevity, we will often omit the domain \(\Omega\) and vector space \(\mathbb{R}^N\) when referring to Bochner spaces \(L^p(0, T; X(\Omega, \mathbb{R}^N))\) throughout this paper. For instance, \(L^p(0, T; X(\Omega, \mathbb{R}^N))\) will be denoted as \(L^p(0, T; X)\).  Throughout the paper the letter $C$ appeared in inequalities will represent some generic positive constants, which are different from time to time.

\section{Convergence to stationary points via  a convex-splitting scheme}\label{section:convergence}
In this section, we introduce the first-order semi-discrete convex-splitting scheme for the second-order flow described by \eqref{eq:damped_hyperbolic_equation}, which is equivalently represented by \eqref{eq:rewritten_damped_hyperbolic1}-\eqref{eq:rewritten_damped_hyperbolic2}. This convex-splitting method is based on the observation that the energy functional \(E\) can be effectively decomposed into the subtraction of two convex functionals, namely $E=E_c-E_e$. A typical decomposition for energy functional \eqref{eq:energy} is given by
\begin{align}\label{eq:decomposition}
E_c = \int_{\Omega} \left( \frac{1}{2} |\nabla \mathbf{u}|^2 + \frac{\alpha}{4}|\mathbf{u}|^4 + \gamma \right) \mathrm{d} \mathbf{x}, \quad E_e = \frac{\beta}{2}\int_{\Omega} |\mathbf{u}|^2 \mathrm{d} \mathbf{x}. 
\end{align}

Building on the decomposition \eqref{eq:decomposition} and fixing a timestep $\tau >0$ and a sequence of damping coefficients $\eta^ k:=\eta(t_k) > 0$ with $t_k=k\tau$, $k=0,1,\ldots$, we introduce the following first-order convex-splitting scheme for \eqref{eq:rewritten_damped_hyperbolic1}-\eqref{eq:rewritten_damped_hyperbolic2}:
\begin{align}
\mathbf{v}^{k+1} - \mathbf{v}^k & = -\tau {\mu}^{k+1} - \tau \eta^{k+1} \mathbf{v}^{k+1}, \label{eq:cs-1} \\
{\mu}^{k+1} & = \delta_{\mathbf{u}} E_c(\mathbf{u}^{k+1}) - \delta_{\mathbf{u}} E_e(\mathbf{u}^k), \label{eq:cs-2} \\
\mathbf{u}^{k+1} - \mathbf{u}^k & = \tau \mathbf{v}^{k+1}. \label{eq:cs-3}
\end{align}
By taking the \(L^2\) inner product of \eqref{eq:cs-1} with \(\mathbf{v}^{k+1}\) and of \eqref{eq:cs-3} with \(\mu^{k+1}\), and then combining these results, the following pseudo-energy stability property arises:
\begin{align}
\mathcal{E}(\mathbf{u}^{k+1}, \mathbf{v}^{k+1}) {\!}-{\!} \mathcal{E}(\mathbf{u}^k, \mathbf{v}^k) 
&\leq \left(\delta_{\mathbf{u}} E_c(\mathbf{u}^{k+1}) - \delta_{\mathbf{u}} E_e(\mathbf{u}^k), \mathbf{u}^{k+1} {\!}-{\!} \mathbf{u}^k\right) {\!}+{\!} \left(\mathbf{v}^{k+1}, \mathbf{v}^{k+1} {\!}-{\!} \mathbf{v}^k\right) {\!\!}\nonumber \\
&= \left({\mu}^{k+1}, \tau \mathbf{v}^{k+1}\right) + \left(\mathbf{v}^{k+1}, \mathbf{v}^{k+1} - \mathbf{v}^k\right) \nonumber \\
&= -\tau \eta^{k+1} \| \mathbf{v}^{k+1} \|_{L^2}^2 \leq 0, \label{eq:pseudo_energy-calculate}
\end{align}
where we use the convexity of $E_c$ and $E_e$ with respect to $\mathbf{u}$, and $\frac{1}{2}\|\mathbf{v}\|_{L^2}^2$ with respect to $\mathbf{v}$. Similarly, we decompose $F$ into $F=F_c-F_e$ with $F_c(\mathbf{u})=\frac{\alpha}{4}|\mathbf{u}|^4 + \gamma$ and $F_e(\mathbf{u})=\frac{\beta}{2}|\mathbf{u}|^2$, and denote \( f_c =  F_c^{\prime} \) and \( f_e = F_e^{\prime} \). To facilitate subsequent analysis, we recast the spatially continuous, temporally discrete scheme \eqref{eq:cs-1}-\eqref{eq:cs-3} as
\begin{align}\label{eq:combined-cs}
    \frac{\mathbf{u}^{k+1} - 2\mathbf{u}^k + \mathbf{u}^{k-1}}{\tau^2} + \eta^{k+1} \frac{\mathbf{u}^{k+1} - \mathbf{u}^k}{\tau} &= \Delta \mathbf{u}^{k+1} -\left(f_c(\mathbf{u}^{k+1}) - f_e(\mathbf{u}^k)\right).
\end{align}
Here we set $\mathbf{v}^0 \equiv \mathbf{0}$, or equivalently, $\mathbf{u}^{-1} \equiv \mathbf{u}^0$.

For the sake of simplicity, we assume throughout this paper that the initial value of $\partial_t \mathbf{u}$ in the PDE \eqref{eq:damped_hyperbolic_equation} is $\partial_t \mathbf{u}(\cdot,0)=\mathbf{v}^0 \equiv \mathbf{0}$. This condition also ensures that the flow trajectory is contained within the energy sub-level set $\{\mathbf{u} \in H: E(\mathbf{u}) \leq E(\mathbf{u}^0)\}$, providing some local stability for energy minimization. This follows directly from the dissipation of the pseudo-energy functional $\mathcal{E}(\mathbf{u},\mathbf{v})$ defined in \eqref{pseudo-energy-ac}. It is worth noting that the findings and conclusions drawn in the following would remain valid if $\mathbf{v}^0\neq \mathbf{0}$ were chosen in appropriate spaces. We further assume the domain $\Omega$ to be of $C^2$ boundary so that $H^2$ regularity of the solutions can be guaranteed. This assumption could be relaxed to convex domains with Lipschitz boundaries by using more refined versions of such estimates (see, e.g., \cite{Gri85}).

We show two useful lemmas first.
\begin{lemma}\label{lemma3:energy-bound}
Let \(\mathbf{u} \in H_0^1(\Omega, \mathbb{R}^N)\). Then
\begin{align*}
E(\mathbf{u})\geq C\left\|\mathbf{u}\right\|_{H^1}^2 +C^{\star},
\end{align*}
where \( C>0 \) and \( C^{\star}\) are constants depending only on $\Omega$, $\alpha$, $\beta$ and $\gamma$.
\end{lemma}

\begin{proof}
     Using Poincar\'{e}'s inequality, there exists a constant $C>0$ such that \( \frac12\|\nabla \mathbf{u}\|_{L^2} \geq C \left\|\mathbf{u}\right\|_{H^1} \). The integral of \(F(\mathbf{u})\) over \(\Omega\) can be rewritten as follows:
\begin{align*}\int_{\Omega}F(\mathbf{u})d\mathbf{x}
    &=\int_{\Omega}\left(\frac{\alpha}{4}|\mathbf{u}|^4-\frac{\beta}{2}|\mathbf{u}|^2+\gamma\right) \mathrm{d} \mathbf{x}
    =\int_{\Omega}\frac{\alpha}{4}\left(|\mathbf{u}|^2-\frac{\beta}{\alpha}\right)^2 + \left(\gamma-\frac{\beta^2}{4\alpha}\right) \mathrm{d} \mathbf{x}.
    \end{align*}
    By taking \(C^{\star}=\int_{\Omega}\left(\gamma-\frac{\beta^2}{4\alpha}\right) \mathrm{d} \mathbf{x}\), we obtain the estimate.
\end{proof}

\begin{lemma}\label{lemma:priori-estimates_h1}
   Assume \( \mathbf{u}^0 \in H_0^1(\Omega, \mathbb{R}^N) \) and that $\eta^k >0$ is bounded above. Let \( \mathbf{u}^k \in H_0^1(\Omega, \mathbb{R}^N) \) and \( \mathbf{v}^k \in L^2(\Omega, \mathbb{R}^N) \) be the solutions of \eqref{eq:cs-1}-\eqref{eq:cs-3} for all \( k \geq 1 \). Then it holds that
    \begin{align}\label{eq:semi-discrete-estimate1}
        \max_{ k\geq 1}\left(\| \mathbf{u}^k \|_{H^1} + \| \mathbf{v}^k \|_{L^2} + \left\| \frac{\mathbf{v}^k - \mathbf{v}^{k-1}}{\tau} \right\|_{H^{-1}}\right) \leq C,
    \end{align}
    where \( C > 0 \) is a constant independent of \( \tau \).
\end{lemma}

\begin{proof}
Applying Lemma~\ref{lemma3:energy-bound} and using the non-increasing property of the pseudo-energy \( \mathcal{E} \) and the choice $\mathbf{v}^0 \equiv 0$, we establish the following inequalities:
    $$
    C\|\mathbf{u}^k\|_{H^1}^2 +C^{\star}+\frac12\|\mathbf{v}^k\|_{L^2}^2\leq E(\mathbf{u}^k)+\frac12\|\mathbf{v}^k\|_{L^2}^2 = \mathcal{E}(\mathbf{u}^k, \mathbf{v}^k) \leq \mathcal{E}(\mathbf{u}^0, \mathbf{v}^0), 
    $$
from which the estimates on the first two terms are obtained. For the third term, \eqref{eq:combined-cs} and $\eta_k$ being bounded above allows us to estimate
    \begin{align}
    \left\|\frac{\mathbf{v}^{k+1} - \mathbf{v}^k}{\tau}\right\|_{H^{-1}}\! 
    & \leq \|\Delta \mathbf{u}^{k+1}\|_{H^{-1}} \!+ \alpha\big\||\mathbf{u}^{k+1}|^2\mathbf{u}^{k+1}\big\|_{H^{-1}} \!+ \beta\|\mathbf{u}^k\|_{H^{-1}} \!+ \eta^{k+1}\|\mathbf{v}^{k+1}\|_{H^{-1}}\!\!\! \nonumber\\
    & \leq \|\nabla \mathbf{u}^{k+1}\|_{L^2} + \alpha\big\||\mathbf{u}^{k+1}|^2\mathbf{u}^{k+1}\big\|_{L^2} + \beta\|\mathbf{u}^k\|_{L^2} + \eta^{k+1}\|\mathbf{v}^{k+1}\|_{L^2} \nonumber\\
    & \leq C\left(\|\mathbf{u}^{k+1}\|_{H^1} + \|\mathbf{u}^{k+1}\|_{H^1}^3+\|\mathbf{u}^k\|_{L^2} + \|\mathbf{v}^{k+1}\|_{L^2}\right),\label{eq:third-estimate-calculation}
    \end{align}
    which concludes the proof.
\end{proof}

The forthcoming analysis, which ensures unique energy-stable solutions for the proposed scheme, is summarized by the following theorem:
\begin{theorem}[Unconditional unique solvability]\label{theorem3:unique-solvability}
    Given any \( \tau > 0 \) and initial conditions \( \mathbf{u}^k, \mathbf{u}^{k-1} \in H_0^1(\Omega, \mathbb{R}^N)\), there exists a unique solution \( \mathbf{u}^{k+1} \in H^2(\Omega, \mathbb{R}^N)\cap H_0^1(\Omega, \mathbb{R}^N)\) for the scheme \eqref{eq:cs-1}-\eqref{eq:cs-3}, or equivalently, the scheme \eqref{eq:combined-cs}. This unique solution ensures the following inequality related to pseudo-energy decay:
    \begin{equation}\label{eq:pseudo_energy-decay}
        \mathcal{E}(\mathbf{u}^{k+1}, \mathbf{v}^{k+1}) + \tau \eta^{k+1} \|\mathbf{v}^{k+1}\|_{L^2}^2 \leq \mathcal{E}(\mathbf{u}^k, \mathbf{v}^k),
    \end{equation}
    where \( \mathcal{E} \) is the pseudo-energy defined in equation \eqref{pseudo-energy-ac}. Moreover, assuming that $\eta^k >0$ is bounded above, these solutions satisfy
    \begin{equation}\label{eq:tau-exploding-h2-estimate}\|\mathbf{u}^{k+1}\|_{H^2} \leq C(\tau) \quad\text{ for all }k\geq 0,\end{equation}
    where $C(\tau)$ depends on $\tau$ but is independent of $k$.
\end{theorem}

\begin{proof}
Starting with $\mathbf{u}^k, \mathbf{u}^{k-1} \in H_0^1(\Omega, \mathbb{R}^N)$, we define
\begin{align*}
 G^k(\mathbf{u}) = \frac{\eta^{k+1}+\frac{1}{\tau}}{2 \tau}\|\mathbf{u}-\mathbf{u}^k\|_{L^{2}}^2 - \frac{1}{\tau}\left(\mathbf{u}, \mathbf{v}^k\right) + \frac{1}{2}\|\nabla \mathbf{u}\|_{L^2}^2 + \int_{\Omega}\left(F_c(\mathbf{u})-\mathbf{u}\cdot f_e(\mathbf{u}^k)\right)\mathrm{d} \mathbf{x}.
\end{align*}
This functional $G^k$ is shown to be strictly convex and coercive on $H_0^1(\Omega, \mathbb{R}^N)$, implying the existence of a unique minimizer, denoted as $\mathbf{u}^{k+1} \in H_0^1(\Omega, \mathbb{R}^N)$. Furthermore, $\mathbf{u}^{k+1}$ is the unique minimizer of $G^k$ if and only if it is the unique weak solution to \eqref{eq:combined-cs}, i.e., for all $\xi \in H_0^1(\Omega, \mathbb{R}^N)$
\begin{equation*}
  \left(\frac{\mathbf{u}^{k+1} - 2\mathbf{u}^k + \mathbf{u}^{k-1}}{\tau^2} + \eta^{k+1}\frac{\mathbf{u}^{k+1} - \mathbf{u}^k}{\tau}, \xi\right) 
  = -a(\mathbf{u}^{k+1}, \xi) - \left(f_c(\mathbf{u}^{k+1}) - f_e(\mathbf{u}^k), \xi\right). 
\end{equation*}
The energy stability property given in the inequality \eqref{eq:pseudo_energy-decay} follows from the calculations detailed in \eqref{eq:pseudo_energy-calculate}. To show $H^2$ regularity, let us rewrite \eqref{eq:combined-cs} in strong form as
\begin{align}-\Delta \mathbf{u}^{k+1} +\frac{1}{\tau^2}\mathbf{u}^{k+1} 
&= -f_c(\mathbf{u}^{k+1}) + f_e(\mathbf{u}^{k}) + \frac{2\mathbf{u}^k - \mathbf{u}^{k-1}}{\tau^2}+ \eta^{k+1}\frac{\mathbf{u}^{k}-\mathbf{u}^{k+1}}{\tau} \nonumber\\
&=: R_\tau(\eta^{k+1},\mathbf{u}^{k-1},\mathbf{u}^{k},\mathbf{u}^{k+1}),
\label{eq:combined-cs-again}
\end{align}
with $R_\tau: \mathbb{R}^+ \times (\mathbb{R}^d)^3 \to \mathbb{R}^d$ satisfying
\[\big| R_\tau(b,\mathbf{w}_1, \mathbf{w}_2, \mathbf{w}_3) \big| \leq C_b(\tau) \left( |\mathbf{w}_1| + |\mathbf{w}_2| +|\mathbf{w}_3|+|\mathbf{w}_3|^3\right)\]
with $C_b(\tau)$ bounded from above for fixed $\tau$ as $b \to 0$. Moreover, given that \( d \leq 3 \), using the Sobolev embeddings 
\[ \|\mathbf{u}^k\|_{L^2}\leq C\|\mathbf{u}^k \|_{H^1}
\quad \text{ and } \quad \|\mathbf{u}^k\|_{L^6} \leq C\|\mathbf{u}^k \|_{H^1} ,  \]
and the estimates in Lemma \ref{lemma:priori-estimates_h1}, we can infer that there exists a constant \( C(\tau) \) depending on \( \tau \) such that
\[ \big\|R_\tau(\eta^{k+1},\mathbf{u}^{k-1},\mathbf{u}^{k},\mathbf{u}^{k+1})\big\|_{L^2} \leq C(\tau) \quad \text{ for all }k\geq 1. \]
Given this, we consider \eqref{eq:combined-cs-again} as an equation for $\mathbf{u}^{k+1}$ with fixed right-hand side given as $R_\tau(\eta^{k+1},\mathbf{u}^{k-1},\mathbf{u}^{k},\mathbf{u}^{k+1})$, which in fact also decouples it into $N$ scalar equations. In that setting noticing that the coefficient $1/\tau^2 $ in \eqref{eq:combined-cs-again} is positive and bounded above uniformly for fixed $\tau$, we apply a standard $H^2$ regularity theorem such as \cite[Sec.~6.3.2, Thm.~4]{evans2010partial} to arrive at \eqref{eq:tau-exploding-h2-estimate}.
\end{proof}

We now prove the result of subsequential convergence of the solutions of \eqref{eq:combined-cs} to a stationary point of $E=E_c - E_e$, which aligns with the goal for introducing the second-order flows. The strategy hinges on exploiting the pseudo-energy decay estimate \eqref{eq:pseudo_energy-decay}. In our result, we are able to handle the case in which $\eta^k$ converges to zero at a speed of $O(1/k)$ or slower. This case is particularly relevant, because when $\eta(t)=3/t$, some of its temporal discretization corresponds to the Nesterov accelerated gradient descent method; see \cite{SuBoyCan16}. Such acceleration and its time-continuous counterpart, when used for general convex energies, lead to optimal convergence rates which are not guaranteed by either gradient descent or the heavy ball method (i.e. with constant damping). In the time-continuous setting and for non-convex energies, stronger results of convergence of the whole trajectory to a stationary point typically hinge on the use of \L{}ojasiewicz-Simon inequalities. A prominent example is the approach in \cite[Thm.~1.2]{HarJen99} which in fact applies (if extended to the vector-valued setting) to global solutions of \eqref{eq:damped_hyperbolic_equation}, but only with $d=2$ and constant damping coefficient $\eta$.

\begin{theorem}[Subsequential convergence to a stationary point]\label{theorem:subsequential-convergence}
For every fixed $\tau>0$, assume that the damping coefficients $\eta^k$ satisfy
    \begin{equation}\label{eq:Nesterov-or-faster}
\frac{\omega}{k} \leq \eta^k \leq \omega_0  \quad \text{ for some } \omega, \omega_0 >0.  
    \end{equation}
Let $\{\mathbf{u}^k\}_k$ be the sequence generated by the first-order semi-implicit scheme \eqref{eq:cs-1}-\eqref{eq:cs-3}.  Then, there exists a subsequence of $\{\mathbf{u}^k\}_{k}$ converges strongly in $H^1(\Omega, \mathbb{R}^N)$ to a stationary point of $E$.
\end{theorem}

\begin{proof}
Rearranging \eqref{eq:pseudo_energy-calculate} as 
\begin{equation}\label{eq:pseudo_energy-calculate-again}\tau \eta^{k+1} \|\mathbf{v}^{k+1}\|^2_{L^2} \leq \mathcal{E}(\mathbf{u}^{k}, \mathbf{v}^{k}) - \mathcal{E}(\mathbf{u}^{k+1}, \mathbf{v}^{k+1}) \quad \text{ for all }k \geq 0,\end{equation}
and summing over these for $k=0,\ldots,n-1$ for any $n \in \mathbb{N}$, we get
\[\tau \sum_{k = 1}^n \eta^k \|\mathbf{v}^k\|^2_{L^2} \leq \mathcal{E}(\mathbf{u}^{0}, \mathbf{v}^{0}) - \mathcal{E}(\mathbf{u}^{n}, \mathbf{v}^{n})\leq \mathcal{E}(\mathbf{u}^{0}, \mathbf{v}^{0})-C^{\star},\]
implying that
\begin{equation}\label{eq:pseudo_energy_decay_series}\sum_{k = 1}^\infty \eta^k \|\mathbf{v}^k\|^2_{L^2} < +\infty.\end{equation}
Now, let us define a sequence of nonnegative numbers by
\begin{equation}\label{eq:Conv_vNorm_sum}c_k := \|\mathbf{v}^k\|^2_{L^2} + \|\mathbf{v}^{k+1}\|^2_{L^2} \quad \text{for all }\, k \geq 1,\end{equation}
and notice that by \eqref{eq:Nesterov-or-faster} we have
\[\frac{c_k}{k+1} = \frac{\|\mathbf{v}^k\|^2_{L^2}}{k+1} + \frac{\|\mathbf{v}^{k+1}\|^2_{L^2}}{k+1} < \frac{\|\mathbf{v}^k\|^2_{L^2}}{k} + \frac{\|\mathbf{v}^{k+1}\|^2_{L^2}}{k+1} \leq \frac{\eta^k}{\omega}\|\mathbf{v}^k\|^2_{L^2} + \frac{\eta^{k+1}}{\omega}\|\mathbf{v}^{k+1}\|^2_{L^2}.\]
This estimate, together with \eqref{eq:pseudo_energy_decay_series}, yields
\begin{equation}\label{eq:gluedsum}\sum_{k=1}^\infty \frac{c_k}{k+1} \leq \frac{2}{\omega} \sum_{k=1}^\infty \eta^k \|\mathbf{v}^k\|^2_{L^2} < + \infty.\end{equation}
Since $c_k\geq 0$, we conclude immediately that \(\underset{k\to \infty}{\liminf}\; c_k = 0\). Hence, we obtain a subsequence \(k_\ell\) such that
\begin{equation*}c_{k_\ell}= \|\mathbf{v}^{k_\ell}\|^2_{L^2} + \|\mathbf{v}^{k_\ell+1}\|^2_{L^2} \xrightarrow[\ell \to \infty]{} 0, \ \text{ so }\  \|\mathbf{v}^{k_\ell}\|^2_{L^2} \xrightarrow[\ell \to \infty]{} 0 \ \text{ and }\  \|\mathbf{v}^{k_\ell+1}\|^2_{L^2} \xrightarrow[\ell \to \infty]{} 0,\end{equation*}
implying that
\begin{equation}\label{eq:conv-v2k-L2}
\mathbf{v}^{k_\ell} \xrightarrow[\ell \to \infty]{} 0 \; \text{ and }\;
\mathbf{v}^{k_\ell+1} \xrightarrow[\ell \to \infty]{} 0 \quad \text{ strongly in }L^2(\Omega,\mathbb{R}^N).
\end{equation}
On the other hand, by \eqref{eq:tau-exploding-h2-estimate}, we have that the sequences $\mathbf{u}^{k_\ell+1}$, $\mathbf{v}^{k_\ell} = (\mathbf{u}^{k_\ell}-\mathbf{u}^{k_\ell-1})/\tau$ and $\mathbf{v}^{k_\ell+1} = (\mathbf{u}^{k_\ell+1}-\mathbf{u}^{k_\ell})/\tau$ are bounded in $H^2(\Omega, \mathbb{R}^N)$. Applying the Banach-Alaoglu theorem and the compact embedding $H^2(\Omega, \mathbb{R}^N) \subset H^1(\Omega, \mathbb{R}^N)$, we can extract a further common subsequence $k_{\ell_m}$ and three limit functions $\mathbf{u}^\infty, \mathbf{v}_0^\infty, \mathbf{v}_1^\infty \in H^2(\Omega, \mathbb{R}^N) \cap H^1_0(\Omega, \mathbb{R}^N)$ such that
\begin{equation}\label{eq:compactness-u-v}
\left.\begin{aligned}
\mathbf{u}^{k_{\ell_m}+1} &\xrightarrow[m \to \infty]{} \mathbf{u}^\infty \,\;\\
\mathbf{v}^{k_{\ell_m}} &\xrightarrow[m \to \infty]{} \mathbf{v}_0^\infty \,\;\\
\mathbf{v}^{k_{\ell_m}+1} &\xrightarrow[m \to \infty]{} \mathbf{v}_1^\infty \,\;   
\end{aligned}\right\}
\quad \text{ strongly in }H^1(\Omega,\mathbb{R}^N),
\end{equation}
which combined with \eqref{eq:conv-v2k-L2} means that we must have $\mathbf{v}_0^\infty = \mathbf{v}_1^\infty = 0$. We can then pass to the limit of this subsequence in the weak formulation of \eqref{eq:cs-1}-\eqref{eq:cs-3} or \eqref{eq:combined-cs} (taking into account the definition of $\mathbf{v}^k$), such that for all $\mathbf{z} \in H^1_0(\Omega, \mathbb{R}^N)$
\begin{equation*}
 \int_\Omega \big( \delta_{\mathbf{u}} E_c(\mathbf{u}^\infty) - \delta_\mathbf{u} E_e(\mathbf{u}^\infty)\big)\cdot \mathbf{z}\, \mathrm{d} \mathbf{x}
 = a\left(\mathbf{u}^\infty,\mathbf{z}\right) +\left(f_c(\mathbf{u}^\infty) - f_e(\mathbf{u}^\infty), \mathbf{z}\right) = 0,
\end{equation*}
which is the stationarity condition for $E$ at $\mathbf{u}^\infty$ with respect to perturbations in $H^1_0(\Omega, \mathbb{R}^N)$.
\end{proof}

As the timestep size approaches zero and with initial values \(\mathbf{u}^0, \mathbf{u}^1 \in H^2(\Omega,\mathbb{R}^N)\), in the next section, we obtain some timestep-independent estimates, thereby proving the well-posedness of the PDE solution they approximate. Given that our goal is to develop a numerical scheme for approaching stationary points of certain functionals, the regularity of initial values is not an issue. Using similar methodology,  in Section~\ref{sec:second-order}, we provide results on energy stability and subsequence convergence to stationary points for the second-order convex splitting scheme, however, other results as stated for the first-order scheme will be omitted for the second-order one.

\section{Timestep-independent estimates and well-posedness of the PDE}\label{sec:well-posedness}
In this section, we prove the existence and uniqueness of a  solution of the second-order flow equation \eqref{eq:damped_hyperbolic_equation}. We provide first an estimate that will be needed to recover solutions of the PDE by letting the timestep $\tau=T/n \to 0^+$ (i.e. $n\to\infty$) for every $T\in\mathbb{R}^+$.

\begin{lemma}\label{lemma:priori-estimates_h2}
Let \( \mathbf{u}^0 \in H^2(\Omega, \mathbb{R}^N) \cap H_0^1(\Omega, \mathbb{R}^N) \).  For any fixed \( T > 0 \), define the time step size as \( \tau = {T}/{n} \), where \( n  \) is a positive integer.  Then, for all \( n \geq 1 \) and \( k = 1, \dots, n \), the following estimate holds:
    \begin{align}\label{eq:semi-discrete-estimate2}
        \| \mathbf{u}^k \|_{H^2} + \| \mathbf{v}^k \|_{H^1} + \left\| \frac{\mathbf{v}^k - \mathbf{v}^{k-1}}{\tau} \right\|_{L^{2}} \leq  C(T),
    \end{align}
    where \( C(T) > 0 \) is a constant independent of \( \tau \) but dependent on \( \| \mathbf{u}^0 \|_{H^2} \) and \( T \).
\end{lemma}
\begin{proof}
We start by taking the inner product of equation \eqref{eq:combined-cs} with \( -\Delta(\mathbf{u}^{k+1}-\mathbf{u}^{k}) \in L^2(\Omega, \mathbb{R}^N)\), which yields:
\begin{align}
-\frac{1}{\tau^2}\left(\mathbf{u}^{k+1}-2\mathbf{u}^k+\mathbf{u}^{k-1}, \Delta(\mathbf{u}^{k+1}-\mathbf{u}^k)\right) - \frac{\eta^{k+1}}{\tau}\left(\mathbf{u}^{k+1}-\mathbf{u}^k, \Delta(\mathbf{u}^{k+1}-\mathbf{u}^k)\right)   &\nonumber \\
+\left(\Delta \mathbf{u}^{k+1}, \Delta(\mathbf{u}^{k+1}-\mathbf{u}^k)\right) = \left(f_c(\mathbf{u}^{k+1})-f_e(\mathbf{u}^k), \Delta(\mathbf{u}^{k+1}-\mathbf{u}^k)\right). & \label{eq:1st-combinition}
\end{align}
We can handle the terms on the left-hand side one by one as follows:
\begin{align}
&-\frac{1}{\tau^2}\left(\mathbf{u}^{k+1}-2\mathbf{u}^k+\mathbf{u}^{k-1}, \Delta(\mathbf{u}^{k+1}-\mathbf{u}^k)\right) \nonumber\\
&\qquad\qquad = \frac{1}{2}\left(\|\nabla \mathbf{v}^{k+1}\|_{L^2}^2 - \|\nabla \mathbf{v}^k\|_{L^2}^2 +\|\nabla (\mathbf{v}^{k+1}-\mathbf{v}^k)\|_{L^2}^2\right), \label{eq:1st-term} \\
&- \frac{\eta^{k+1}}{\tau}\left(\mathbf{u}^{k+1}-\mathbf{u}^k, \Delta(\mathbf{u}^{k+1}-\mathbf{u}^k)\right) = \tau \eta^{k+1}\|\nabla \mathbf{v}^{k+1}\|_{L^2}^2, \label{eq:2nd-term} \\
&\left(\Delta \mathbf{u}^{k+1}, \Delta(\mathbf{u}^{k+1}-\mathbf{u}^k)\right) \geq \frac{1}{2}\left(\|\Delta \mathbf{u}^{k+1}\|_{L^2}^2 - \|\Delta \mathbf{u}^{k}\|_{L^2}^2\right). \label{eq:3rd-term}
\end{align}
For the nonlinear terms, we have
\begin{align}
&\left(f_c(\mathbf{u}^{k+1}) - f_e(\mathbf{u}^{k}), \Delta(\mathbf{u}^{k+1}-\mathbf{u}^k)\right) \nonumber\\
&\qquad\qquad = \tau\left(\alpha\left(|\mathbf{u}^{k+1}|^2\nabla \mathbf{u}^{k+1}+\mathbf{u}^{k+1}(\mathbf{u}^{k+1})^T\nabla \mathbf{u}^{k+1}\right) - \beta\nabla \mathbf{u}^{k}, \nabla \mathbf{v}^{k+1}\right) \nonumber \\
&\qquad\qquad \leq C\tau\|\mathbf{u}^{k+1}\|_{L^\infty}^4\|\nabla \mathbf{u}^{k+1}\|_{L^2}^2 + C\tau\|\nabla \mathbf{u}^{k}\|_{L^2}^2 + \tau\|\nabla \mathbf{v}^{k+1}\|_{L^2}^2. \label{eq:4th-term1}
\end{align}
Utilizing the Gagliardo-Nirenberg inequalities (see, e.g., \cite[Thm.~5.9]{AdaFou03}), 
\begin{align*}
\|\mathbf{u}\|_{L^\infty} &\leq C\|\mathbf{u}\|_{H^2}^{1/2}\|\mathbf{u}\|_{L^6}^{1/2},\quad d=3, \\
\|\mathbf{u}\|_{L^\infty} &\leq C\|\mathbf{u}\|_{H^2}^{1/2}\|\mathbf{u}\|_{L^2}^{1/2},\quad d=2, \\
\|\mathbf{u}\|_{L^\infty} &\leq C\|\mathbf{u}\|_{H^1}^{1/2}\|\mathbf{u}\|_{L^2}^{1/2},\quad d=1, 
\end{align*}
and the Sobolev embeddings, we obtain
\begin{align*}
\|\mathbf{u}^{k+1}\|_{L^\infty}^4 \leq C \|\mathbf{u}^{k+1}\|_{H^2}^{2}\|\mathbf{u}^{k+1}\|_{H^1}^{2}\quad (d=1,2,3). 
\end{align*}

Incorporating this into \eqref{eq:4th-term1}, considering the uniform boundedness of \(\|\mathbf{u}^{k}\|_{H^1}\) and \(\|\mathbf{u}^{k+1}\|_{H^1}\), and using the $H^2$ regularity estimate
\begin{equation*}
\|\mathbf{u}^{k+1}\|_{H^2} \leq C\big(\|\Delta \mathbf{u}^{k+1}\|_{L^2} + \|\mathbf{u}^{k+1}\|_{L^2}\big),
\end{equation*}
we derive
\begin{align}
\left(f_c(\mathbf{u}^{k+1}) - f_e(\mathbf{u}^{k}), \Delta(\mathbf{u}^{k+1}-\mathbf{u}^k)\right) \leq C\tau\|\Delta \mathbf{u}^{k+1}\|_{L^2}^2 + \tau\|\nabla \mathbf{v}^{k+1}\|_{L^2}^2 + C\tau. \label{eq:4th-term}
\end{align}

Defining $F_1^k := \frac{1}{2}\|\Delta \mathbf{u}^k\|_{L^2}^2 + \frac{1}{2}\|\nabla \mathbf{v}^k\|_{L^2}^2$,
and consolidating \eqref{eq:1st-combinition}-\eqref{eq:4th-term}, we obtain
$$
F_1^{k+1} - F_1^k \leq C\tau\|\Delta\mathbf{u}^{k+1}\|_{L^2}^2 + \tau\|\nabla \mathbf{v}^{k+1}\|_{L^2}^2 + C\tau
\leq C\tau F_1^{k+1} +C\tau.
$$
Summing over \(k\) and recognizing that \(F_1^0 \leq C\left\|\mathbf{u}^0\right\|_{H^2}^2\) (given \( \mathbf{v}^0 \equiv 0 \)), we infer
\begin{align*}
F_1^{k} \leq C\left(\left\|\mathbf{u}^0\right\|_{H^2}^2 +T\right) + C\tau \sum_{j=1}^{k} F_1^{j}.
\end{align*}
Utilizing the discrete Gronwall inequality (see, e.g., \cite[Lem.~100]{Dra03}) yields $F_1^{k} \leq C(T)$. This substantiates the first two estimates in Equation \eqref{eq:semi-discrete-estimate2}. The third estimate emerges from a calculation analogous to that in \eqref{eq:third-estimate-calculation}.
\end{proof}

Armed with the estimates on the semi-discrete flows, we now prove the existence and uniqueness of a global strong solution to the continuous second-order flow. 

\begin{theorem}\label{theorem:strong-wellposedness}
Let \(  \mathbf{u}^0 \in H^2(\Omega, \mathbb{R}^N)\cap H_0^1(\Omega, \mathbb{R}^N) \) and $T \in \mathbb{R}^+$. Then with initial conditions \( \mathbf{u}(\cdot, 0) \equiv \mathbf{u}^0 \) and \( \partial_t \mathbf{u}(\cdot, 0) \equiv \mathbf{0} \), there exists a unique
\begin{equation}
    \mathbf{u} \in L^{\infty}(0, T; H^2\cap H_0^1)\cap W^{1,\infty}(0, T; H^1) \cap W^{2,\infty}(0, T; L^{2})
\end{equation}
which satisfies equation \eqref{eq:damped_hyperbolic_equation}.
\end{theorem}
\begin{proof}
{\bf Existence:} The proof of existence utilizes the basic idea of Rothe’s method \cite{rot30}, which is based on constructing numerical solutions and leveraging the corresponding a priori estimates. 

Define \( \tau = {T}/{n} \), where \( n \) is a positive integer, and let \( t_k = k\tau \) for \( k = 0, \ldots, n \). For each \( k = 0, \ldots, n-1 \), let \( \mathbf{u}^{k+1} \) denote the unique solution of the semi-discrete scheme \eqref{eq:cs-1}-\eqref{eq:cs-3}. Define \(\mathbf{v}^{k+1} = (\mathbf{u}^{k+1} - \mathbf{u}^{k}) / \tau\) and \(\mathbf{a}^{k+1} = (\mathbf{v}^{k+1} - \mathbf{v}^{k}) / \tau\) and construct the Rothe sequences \(\mathbf{u}_{n}\) and \(\mathbf{v}_{n}\) as follows:
\[
\begin{aligned}
\mathbf{u}_{n}(\mathbf{x}, t) &:= 
\begin{cases} 
\mathbf{u}^{k}(\mathbf{x}) + (t - t_k) \mathbf{v}^{k+1}, & t_k \leq t < t_{k+1}, \\ 
\mathbf{u}^{n}(\mathbf{x}), & t = t_{n},
\end{cases} \\
\mathbf{v}_{n}(\mathbf{x}, t) &:= 
\begin{cases} 
\mathbf{v}^{k}(\mathbf{x}) + (t - t_k) \mathbf{a}^{k+1}, & t_k \leq t < t_{k+1}, \\ 
\mathbf{v}^{n}(\mathbf{x}), & t = t_{n}.
\end{cases}
\end{aligned}
\]
Here, \( \mathbf{u}_n(\mathbf{x}, t) \) and \( \mathbf{v}_n(\mathbf{x}, t) \) represent piecewise linear interpolations of the numerical solution sequences. Define the piecewise constant (in time) interpolants $\bar{\mathbf{u}}_n(\mathbf{x}, t)$, $\bar{\mathbf{v}}_n(\mathbf{x}, t)$, $\bar{\mathbf{a}}_n(\mathbf{x}, t)$, $\bar{\eta}_n(t)$ such that
\[ \bar{\gamma}_n = \gamma^{k+1} \;\;\mbox{if}\;\; t\in(t_k,t_{k+1}],\;\; k=0,\ldots,n-1,\quad \mbox{for}\; \gamma=\mathbf{u},\mathbf{v},\mathbf{a},\eta. \]
Additionally, we define $\hat{\mathbf{u}}_n(\mathbf{x}, t)$ such that 
\[\hat{\mathbf{u}}_n(\mathbf{x}, t) =\mathbf{u}^{k}(\mathbf{x})\;\;\mbox{if}\;\; t\in(t_k,t_{k+1}],\;\; k=0,\ldots,n-1.\]
This leads to the relationships $\frac{\partial^{-}}{\partial t}\mathbf{u}_{n}(\cdot,t)=\bar{\mathbf{v}}_{n}(\cdot,t)$ and $\frac{\partial^{-}}{\partial t}\mathbf{v}_{n}(\cdot,t)=\bar{\mathbf{a}}_{n}(\cdot,t)$ for all $t\in (0,T)$, where $\frac{\partial^{-}}{\partial t}$ denotes the left-hand derivative. Based on these definitions, the weak formulation of \eqref{eq:combined-cs} yields
\begin{align}\label{eq:piecewise_cs}
\left(\frac{\partial^{-}}{\partial t}\mathbf{v}_{n}(\cdot,t) + \bar{\eta}_{n}(t) \bar{\mathbf{v}}_{n}(\cdot,t) +f_c(\bar{\mathbf{u}}_{n}(\cdot,t)) - f_e(\hat{\mathbf{u}}_{n}(\cdot,t)), \mathbf{z} \right) + a(\bar{\mathbf{u}}_{n}(\cdot,t), \mathbf{z})  = 0,
\end{align}
for all $\mathbf{z}\in H_0^1(\Omega,\mathbb{R}^N)$, applicable for all $n\in\mathbb{Z}^+$ and $t\in (0,T)$. Drawing upon Lemma~\ref{lemma:priori-estimates_h2}, and the definition of $\bar{\mathbf{u}}_{n}$, $\hat{\mathbf{u}}_{n}$ and $\bar{\mathbf{v}}_{n}$, we establish their uniform boundedness as follows: 
\begin{align}
& \left\|\mathbf{u}_{n}(\cdot,t)\right\|_{H^2}+\left\|\mathbf{v}_{n}(\cdot,t)\right\|_{H^1}\leq C, \label{eq:theorem_estimate1} \\
& \|\bar{\mathbf{u}}_{n}(\cdot,t)\|_{H^2}+\|\hat{\mathbf{u}}_{n}(\cdot,t)\|_{H^2}+|\bar{\eta}_{n}|\left\|\bar{\mathbf{v}}_{n}(\cdot,t)\right\|_{H^1}+\left\|\bar{\mathbf{a}}_{n}(\cdot,t)\right\|_{L^2} \leq C, \label{eq:theorem_estimate2} \\
& \left\|\mathbf{u}_{n}(\cdot,t)-\bar{\mathbf{u}}_n(\cdot,t)\right\|_{H^1}+\left\|\mathbf{u}_{n}(\cdot,t)-\hat{\mathbf{u}}_n(\cdot,t)\right\|_{H^1}+\left\|\mathbf{v}_n(\cdot,t)-\bar{\mathbf{v}}_n(\cdot,t)\right\|_{L^2} \label{eq:theorem_estimate3}\\
&\qquad\qquad\qquad\qquad\qquad\qquad+\left\||\mathbf{u}_{n}(\cdot,t)|^2\mathbf{u}_{n}(\cdot,t)-|\bar{\mathbf{u}}_n(\cdot,t)|^2\bar{\mathbf{u}}_n(\cdot,t)\right\|_{L^2}\leq \frac{C}{n}, \nonumber 
\end{align}
for all $n\in\mathbb{Z}^+$, $t\in(0,T)$, as well as
\begin{equation}
\left\|\mathbf{u}_n(\cdot,t)-\mathbf{u}_n\left(\cdot,t^{\prime}\right)\right\|_{H^1}+\left\|\mathbf{v}_n(\cdot,t)-\mathbf{v}_n\left(\cdot,t^{\prime}\right)\right\|_{L^2} \leq C\left|t-t^{\prime}\right|,\label{eq:theorem_estimate4} 
\end{equation}
for all $n\in\mathbb{Z}^+$, $t,t^{\prime}\in(0,T)$. Using methods similar to those in the proof of \cite[Lemma 3]{kav84} with estimates \eqref{eq:theorem_estimate1}-\eqref{eq:theorem_estimate3}, and Gronwall's inequality, we obtain the following for arbitrary $n_1,\;n_2\in\mathbb{Z}^+$:
\begin{equation}\label{eq:uniform}
\left\|\mathbf{u}_{n_1}(\cdot,t)-\mathbf{u}_{n_2}\left(\cdot,t\right)\right\|_{H^1}+\left\|\mathbf{v}_{n_1}(\cdot,t)-\mathbf{v}_{n_2}\left(\cdot,t\right)\right\|_{L^2} \leq C_1
\left(\frac{1}{n_1}+\frac{1}{n_2}\right)\exp(C_2 T),
\end{equation}
where $C_1,\; C_2$ are some constants independent of $n_1$ and $n_2$.
Since $\{\mathbf{u}_n\}_{n\in\mathbb{Z}^+}\subset C(0, T; H_0^1)$ and $\{\mathbf{v}_n\}_{n\in\mathbb{Z}^+} \subset C(0, T; L^2)$, using \eqref{eq:uniform} we derive the uniform convergence of both $\{\mathbf{u}_n\}_{n\in\mathbb{Z}^+}$ and $\{\mathbf{v}_n\}_{n\in\mathbb{Z}^+}$ such that
\begin{align}\label{eq:converge_sequence}
\mathbf{u}_n \rightarrow \mathbf{u} \quad \text { in } \quad C(0, T; H_0^1), \quad \mathbf{v}_n \rightarrow \mathbf{v}\quad \text { in } \quad C(0, T; L^2).
\end{align}
From \eqref{eq:theorem_estimate2}, we deduce that \(\mathbf{u} \in L^\infty(0, T; H^2)\). By passing to the limit as \(n \to \infty\) in \eqref{eq:theorem_estimate4}, we have:
\[
\left\|\mathbf{v}(\cdot,t) - \mathbf{v}(\cdot,t')\right\|_{L^2}^2 + \left\|\mathbf{u}(\cdot,t) - \mathbf{u}(\cdot,t')\right\|_{H^1}^2 \leq C|t - t'|,
\]
which implies \(\frac{\partial \mathbf{v}}{\partial t} \in L^{\infty}(0,T; L^2)\) and \(\frac{\partial \mathbf{u}}{\partial t} \in L^{\infty}(0,T; H^1)\). 
Next, by passing to the limit as \(n \to \infty\) in the following identity 
\[
\mathbf{u}_n(\cdot,t) = \int_0^t \frac{\partial \mathbf{u}_n(\cdot,s)}{\partial t} \mathrm{d}s + \mathbf{u}_0 = \int_0^t \bar{\mathbf{v}}_n(\cdot,s) \mathrm{d}s + \mathbf{u}_0,
\]
we find that \(\frac{\partial \mathbf{u}}{\partial t} =  \mathbf{v} \in L^{\infty}(0,T; H^1)\), and \(\frac{\partial^2 \mathbf{u}}{\partial t^2} = \frac{\partial \mathbf{v}}{\partial t} \in L^{\infty}(0,T; L^2)\). Therefore, we conclude:
\[
\mathbf{u} \in L^{\infty}(0, T; H^2 \cap H_0^1) \cap W^{1,\infty}(0, T; H^1) \cap W^{2,\infty}(0, T; L^2).
\]
The rest is to argue that $\mathbf{u}$ is a solution of \eqref{eq:damped_hyperbolic_equation}. Integrating equation \eqref{eq:piecewise_cs} over the time interval $(0, t)$, we obtain for all $\mathbf{z}\in  H_0^1(\Omega, \mathbb{R}^N)$,
{\small
$$
\left(\mathbf{v}_{n}(t), \mathbf{z}\right)-\left(\mathbf{v}^{0}, \mathbf{z}\right)+\int_0^t \bar{\eta}_{n}(s)\left( \bar{\mathbf{v}}_{n}(s), \mathbf{z}\right)+a\left( \bar{\mathbf{u}}_n(s), \mathbf{z}\right) +\left(\alpha |\bar{\mathbf{u}}_{n}(s)|^2\bar{\mathbf{u}}_{n}(s) - \beta\hat{\mathbf{u}}_{n}(s), \mathbf{z}\right) \mathrm{d} s=0.
$$}
Together with \eqref{eq:theorem_estimate3} and \eqref{eq:converge_sequence}, and passing to the limit as $n \rightarrow \infty$, we have
\begin{align}\label{eq:sequence_limit}
\left(\mathbf{v}(t), \mathbf{z}\right)-\left(\mathbf{v}^{0}, \mathbf{z}\right)+\int_0^t \eta(s)\left( \mathbf{v}(s), \mathbf{z}\right)+a\left( \mathbf{u}(s), \mathbf{z}\right) +\left(\alpha |\mathbf{u}(s)|^2\mathbf{u}(s) - \beta\mathbf{u}(s), \mathbf{z}\right) \mathrm{d} s=0.
\end{align}
Differentiating equation \eqref{eq:sequence_limit} with respect to time $t$, we obtain equation \eqref{eq:damped_hyperbolic_equation}. 
Notably, $\mathbf{u}^0=\mathbf{u}_{n}(\cdot, 0)\rightarrow \mathbf{u}(\cdot, 0)$ and $\mathbf{0}=\mathbf{v}^0=\mathbf{v}_{n}(\cdot, 0)\rightarrow \partial_t\mathbf{u}(\cdot, 0)$, ensuring that the initial condition is satisfied.

{\bf Uniqueness:}  Let \( \mathbf{u}^{(1)} \) and \( \mathbf{u}^{(2)} \) be two strong solutions with the same initial data. Setting \( \tilde{\mathbf{u}} = \mathbf{u}^{(1)} - \mathbf{u}^{(2)} \), we substitute \( \mathbf{u}^{(1)} \) and \( \mathbf{u}^{(2)} \) into equation \eqref{eq:damped_hyperbolic_equation} separately and then subtract the resultant equations to generate a new equation on \( \tilde{\mathbf{u}} \).  By taking the inner product of this equation with the test function \( \tilde{\mathbf{v}} = \partial_t \tilde{\mathbf{u}} \), followed by integration by parts and constraining the nonlinearity through uniform bounds, we reach a formulation which allows for applying Gronwall's inequality. Given that both \( \tilde{\mathbf{u}} \) and \( \tilde{\mathbf{v}} \) start from zero, it dictates that \( \tilde{\mathbf{u}} \) remains identically zero throughout, thus proving the uniqueness. For the detailed steps and calculations for this proof we refer to \cite[Thm. 5.3]{WanWis10}, where similar arguments and techniques are employed in a closely related context.
\end{proof}

\begin{remark}
The theoretical analysis presented in this section also holds for inhomogeneous Dirichlet boundary conditions on \( \partial\Omega \). In that case, we consider the solution space to be an affine manifold \( H_0^1(\Omega, \mathbb{R}^N) + \{\mathbf{u}_b\} \), where \( \mathbf{u}_b \in H^1(\Omega, \mathbb{R}^N) \) is any function satisfying these boundary conditions.
\end{remark}

\section{A second-order semi-discrete convex-splitting scheme}
\label{sec:second-order}
For numerical computations, a straightforward first-order time stepping scheme like \eqref{eq:cs-1}-\eqref{eq:cs-3} is not the optimal choice due to its limited accuracy. A more effective alternative is to use second-order schemes, which may provide better accuracy without significantly increasing the computational complexity. Here, we introduce the second-order semi-discrete version of a convex-splitting scheme, which is of the following form: 
\begin{align}
\mathbf{v}^{k+1} - \mathbf{v}^k & = -\tau {\mu}^{k+\frac12} - \tau \eta^{k+\frac12} \mathbf{v}^{k+\frac12}, \label{eq:2nd-cs-1} \\
{\mu}^{k+\frac12} & = -\Delta\mathbf{u}^{k+\frac12}+\alpha\chi(\mathbf{u}^{k+1}, \mathbf{u}^k)- \beta\widetilde{\mathbf{u}}^{k+\frac12}, \label{eq:2nd-cs-2} \\
\mathbf{u}^{k+1} - \mathbf{u}^k & = \tau \mathbf{v}^{k+\frac12}, \label{eq:2nd-cs-3}
\end{align}
where
\begin{align}
\label{eq:u-k+1/2}
\begin{aligned}
&\mathbf{v}^{k+\frac{1}{2}}=\frac{\mathbf{v}^{k}+\mathbf{v}^{k+1}}{2},\quad\mathbf{u}^{k+\frac{1}{2}}=\frac{\mathbf{u}^{k}+\mathbf{u}^{k+1}}{2},\quad \widetilde{\mathbf{u}}^{k+\frac{1}{2}} = \frac{3}{2} \mathbf{u}^{k} - \frac{1}{2} \mathbf{u}^{k-1},\\
& \eta^{k+\frac12}=\eta(t_{k+\frac12})=\eta(\tau(k+1/2)),\quad\chi(\mathbf{u}^{k+1}, \mathbf{u}^k)=\frac12\left(\left|\mathbf{u}^{k+1}\right|^2+\left|\mathbf{u}^{k}\right|^2\right)\mathbf{u}^{k+\frac12}.
\end{aligned}
\end{align}
Here, \(\chi(\mathbf{u}^{k+1}, \mathbf{u}^k)\) is an approximation for \(\big|\mathbf{u}^{k+\frac{1}{2}}\big|^2\mathbf{u}^{k+\frac{1}{2}}\), as typically done for nonlinear problems, see, e.g. \cite{DuNic91,HuWisWan09}. Next, we present some fundamental results for this second-order scheme, with the focus on the convergence to a stationary point of the non-convex variational problem, though they are analogous to those for the first-order scheme.

\begin{theorem}\label{thm:2nd-unique-solvability}  Given any \( \tau > 0 \) and initial conditions \( \mathbf{u}^k, \mathbf{u}^{k-1} \in H_0^1(\Omega, \mathbb{R}^N)\), there exists a unique solution \( \mathbf{u}^{k+1} \in H^2(\Omega, \mathbb{R}^N)\cap H_0^1(\Omega, \mathbb{R}^N)\) for the scheme \eqref{eq:2nd-cs-1}-\eqref{eq:2nd-cs-3}. This unique solution ensures the following energy stability:
{\small\begin{align}\label{eq:modified_pseudo_energy-decay}
   \widetilde{\mathcal{E}}\left(\mathbf{u}^{k+1}, \mathbf{u}^k, \mathbf{v}^{k+1}\right) + \tau \eta^{k+\frac{1}{2}} \left\|\mathbf{v}^{k+\frac{1}{2}}\right\|_{L^2}^2 + \frac{\beta}{4} \left\|\mathbf{u}^{k+1} - 2\mathbf{u}^k + \mathbf{u}^{k-1}\right\|_{L^2}^2 = \widetilde{\mathcal{E}}\left(\mathbf{u}^k, \mathbf{u}^{k-1}, \mathbf{v}^k\right),
\end{align}}%
where
\[
\widetilde{\mathcal{E}}\left(\mathbf{u}^j, \mathbf{u}^{j-1}, \mathbf{v}^j\right) := \mathcal{E}\left(\mathbf{u}^j, \mathbf{v}^j\right) + \frac{\beta}{4} \left\|\mathbf{u}^j - \mathbf{u}^{j-1}\right\|_{L^2}^2.
\]
Moreover, assuming that $\eta^{k+\frac12} >0$ is bounded above, these solutions satisfy
    \begin{equation}\label{eq:2nd-tau-exploding-h2-estimate}
    \|\mathbf{u}^{k+1}\|_{H^2} \leq C(\tau) \quad\text{ for all }\,k\geq 1,\end{equation}
    with $C(\tau)$ independent of $k$.
\end{theorem}
\begin{proof}
The proof of unique solvability for \eqref{eq:2nd-cs-1}-\eqref{eq:2nd-cs-3} is analogous to \cite[Thm.~3.2]{BasHuLow13}. The energy stability \eqref{eq:modified_pseudo_energy-decay} can be shown in a manner similar to \cite[Thm.~3.4]{BasHuLow13}. The regularity result \eqref{eq:2nd-tau-exploding-h2-estimate} we refer to the techniques used in proving \eqref{eq:tau-exploding-h2-estimate} in Theorem \ref{theorem3:unique-solvability} for the first-order scheme.
\end{proof}

\begin{lemma}\label{lemma:2nd-priori-estimates_h1}
   Assume \( \mathbf{u}^0 \in H_0^1(\Omega, \mathbb{R}^N) \) and that $\eta^k >0$ is bounded from above. Let \( \mathbf{u}^k \in H_0^1(\Omega, \mathbb{R}^N) \) and \( \mathbf{v}^k \in L^2(\Omega, \mathbb{R}^N) \) be the solutions of \eqref{eq:2nd-cs-1}-\eqref{eq:2nd-cs-3} for all \( k \geq 1 \). Then the following estimate holds
    \begin{align}\label{eq:2nd-semi-discrete-estimate1}
        \max_{ k\geq 1}\left(\| \mathbf{u}^k \|_{H^1} + \| \mathbf{v}^k \|_{L^2} + \left\| \frac{\mathbf{v}^k - \mathbf{v}^{k-1}}{\tau} \right\|_{H^{-1}}\right) \leq C,
    \end{align}
    where \( C > 0 \) is a constant independent of \( \tau \).
\end{lemma}

\begin{proof} The proof follows similarly to the one for Lemma 
\ref{lemma:priori-estimates_h1}. We omit the details here for brevity. 
\end{proof}

\begin{theorem} [Subsequential convergence to a stationary point] \label{theorem:2nd-subsequential-convergence}
For any fixed $\tau > 0$, assume the damping coefficients $\eta^{k+\frac12}$ satisfy
\begin{equation}\label{eq:2nd-Nesterov-or-faster}
\frac{\omega}{k} \leq \eta^{k+\frac12} \leq \omega_0  \quad \text{ for some } \omega,\; \omega_0 >0.
\end{equation} Let $\{\mathbf{u}^k\}_k$ be the sequence generated by the second-order semi-implicit scheme \eqref{eq:2nd-cs-1}-\eqref{eq:2nd-cs-3}. Then, there exists a subsequence of $\{\mathbf{u}^k\}_k$ that converges strongly in $H^1(\Omega, \mathbb{R}^N)$ to a stationary point of $E$.
\end{theorem}
\begin{proof}
The proof uses the same techniques as those used in Theorem \ref{theorem:subsequential-convergence} for the first-order scheme. For completeness, we include the detailed steps here.

First, for any \(n \in \mathbb{N}\), summing \eqref{eq:modified_pseudo_energy-decay} from \( k = 0 \) to \( n - 1 \), and applying Lemma \ref{lemma:2nd-priori-estimates_h1} we obtain:
\[\tau \sum_{k = 0}^{n-1} \eta^{k+\frac12} \|\mathbf{v}^{k+\frac12}\|^2_{L^2} \leq \widetilde{\mathcal{E}}\left(\mathbf{u} ^{0}, \mathbf{u} ^{-1}, \mathbf{v} ^{0}\right) - \widetilde{\mathcal{E}}\left(\mathbf{u} ^{n}, \mathbf{u} ^{n-1}, \mathbf{v} ^{n}\right)\leq \mathcal{E}(\mathbf{u}^{0} , \mathbf{v}^{0})+ C,\]
which implies that
\begin{equation}\label{eq:2nd-pseudo_energy_decay_series}\sum_{k = 0}^\infty \eta^{k+\frac12} \|\mathbf{v}^{k+\frac12}\|^2_{L^2} < +\infty.\end{equation}
Next, we define \(d_k\) to bundle three adjacent indices together as follows:
\[
d_k := \|\mathbf{v}^{k-\frac{3}{2}}\|^2_{L^2} + \|\mathbf{v}^{k-\frac{1}{2}}\|^2_{L^2} + \|\mathbf{v}^{k+\frac{1}{2}}\|^2_{L^2}, \quad \text{for all } k \geq 2.
\]
Using \eqref{eq:2nd-Nesterov-or-faster} along with \eqref{eq:2nd-pseudo_energy_decay_series}, and applying the reasoning from the proof of Theorem \ref{theorem:subsequential-convergence}, we conclude that \(\liminf_k d_k = 0\). Consequently, there exists a subsequence \( k_\ell \) such that $d_{k_\ell}\to 0$ ($\ell \to \infty$), which implies that
\begin{equation}\label{eq:2nd-conv-v2k-L2}
\mathbf{v}^{k_\ell-\frac{3}{2}} \xrightarrow[\ell \to \infty]{} 0, \; \mathbf{v}^{k_\ell-\frac{1}{2}} \xrightarrow[\ell \to \infty]{} 0, \; \text{and} \;\; \mathbf{v}^{k_\ell+\frac{1}{2}} \xrightarrow[\ell \to \infty]{} 0 \quad \text{strongly in } L^2(\Omega, \mathbb{R}^N).
\end{equation}
Given the \(H^2\) boundedness of \(\mathbf{u}^{k_\ell+1}\), \(\mathbf{v}^{k_\ell-\frac{3}{2}}\), \(\mathbf{v}^{k_\ell-\frac{1}{2}}\), and \(\mathbf{v}^{k_\ell+\frac{1}{2}}\) from \eqref{eq:2nd-tau-exploding-h2-estimate}, and using the Banach-Alaoglu theorem along with the compact embedding \(H^2 \subset H^1\), we can extract a subsequence \(\{k_{\ell_m}\}\) such that
\begin{equation}\label{eq:2nd-compactness-u-v}
\begin{aligned}
&\mathbf{u}^{k_{\ell_m}+1} \xrightarrow[m \to \infty]{} \mathbf{u}^\infty, \quad \mathbf{v}^{k_{\ell_m}-\frac{3}{2}}\xrightarrow[m \to \infty]{} 0, \\
& \mathbf{v}^{k_{\ell_m}-\frac{1}{2}}\xrightarrow[m \to \infty]{} 0, \quad \mathbf{v}^{k_{\ell_m}+\frac{1}{2}} \xrightarrow[m \to \infty]{} 0, 
\end{aligned}
\quad \text{strongly in } H^1(\Omega, \mathbb{R}^N).
\end{equation}
implying that
\begin{align}\label{eq:limit-u}
    \mathbf{u}^{k_{\ell_m}-2},\, \mathbf{u}^{k_{\ell_m}-1}, \,\mathbf{u}^{k_{\ell_m}}, \,\mathbf{u}^{k_{\ell_m}+1} \xrightarrow[m \to \infty]{} \mathbf{u}^\infty \quad \text{strongly in } H^1(\Omega, \mathbb{R}^N).
\end{align}
Next, using the convergence of all four sequences in two consecutive instances of \eqref{eq:2nd-cs-2}, we find
\begin{align}\label{eq:limit-mu}
\int_\Omega\left(\mu^{k_{\ell_m}-\frac12} - \mu^{k_{\ell_m}+\frac12}\right)\cdot\mathbf{z} \, \mathrm{d} \mathbf{x} \xrightarrow[m \to \infty]{} 0 \quad \text{for all }\mathbf{z} \in H^1_0(\Omega, \mathbb{R}^N).
\end{align}
Combining \eqref{eq:2nd-cs-1} and \eqref{eq:2nd-cs-3}, we have
\begin{align}
\left(\frac2\tau-\eta^{k+\frac12}\right)\left(\mathbf{u}^{k+1} - \mathbf{u}^k\right)=\tau \mu^{k+\frac{1}{2}}+2\mathbf{v}^{k+1}. \label{eq:rewrite_2nd-cs-3}
\end{align}
Applying \eqref{eq:limit-u} on the left-hand side of \eqref{eq:rewrite_2nd-cs-3}, we find that for all $\mathbf{z} \in H^1_0(\Omega, \mathbb{R}^N)$,
\begin{align}\label{eq:limit-mu-v}
\int_\Omega\left(\frac\tau2\mu^{k_{\ell_m}-\frac12} + \mathbf{v}^{k_{\ell_m}}\right)\cdot\mathbf{z} \, \mathrm{d} \mathbf{x},\; \int_\Omega\left(\frac\tau2\mu^{k_{\ell_m}+\frac12} + \mathbf{v}^{k_{\ell_m}+1}\right)\cdot\mathbf{z} \, \mathrm{d} \mathbf{x} \xrightarrow[m \to \infty]{} 0.
\end{align}
Now combining \eqref{eq:limit-mu} and \eqref{eq:limit-mu-v}, we obtain
\begin{align}\label{eq:limit-v}
\int_\Omega\left(\mathbf{v}^{k_{\ell_m}+1} - \mathbf{v}^{k_{\ell_m}}\right)\cdot\mathbf{z}\,\mathrm{d} \mathbf{x} \xrightarrow[m \to \infty]{} 0 \quad \text{for all }\mathbf{z} \in H^1_0(\Omega, \mathbb{R}^N).
\end{align}
Finally, passing to the limit with $k_{\ell_m}$ in the weak formulation of \eqref{eq:2nd-cs-1}-\eqref{eq:2nd-cs-3} results in
\[\begin{aligned}\int_\Omega \delta_{\mathbf{u}} E(\mathbf{u}^\infty) \cdot\mathbf{z} \, \mathrm{d} \mathbf{x} = 0 \quad \text{for all }\mathbf{z} \in H^1_0(\Omega, \mathbb{R}^N).\end{aligned}\]
\end{proof}

Similar results like timestep-independent estimates for \(\mathbf{u}^k\) and \(\mathbf{v}^k\) can be obtained as well for the second-order semi-discrete scheme \eqref{eq:2nd-cs-1}-\eqref{eq:2nd-cs-3}. However, we will not discuss it further here.

\section{Applications and numerical results}\label{section:numerics}

In this section, we test the effectiveness and efficiency of the proposed numerical algorithms within the framework of second-order flows, motivated by practical applications. For experiments involving second-order flows, we always adopt \(\eta = 3/t\) as the damping coefficient in our tests. As discussed in Section~\ref{section:convergence}, this choice aligns with the time-continuous version of Nesterov accelerated gradient method. All experiments were conducted on 2D square domains, using bilinear elements for spatial discretization. The computational tests were performed on a workstation equipped with a 2.0GHz CPU (X86-2A2) with 8 cores. All codes were implemented in MATLAB without parallel implementation.

Building on the semi-discrete schemes discussed in the previous section, we introduce two fully discrete convex-splitting schemes for second-order flow.  Here, \( S_h \) denotes the finite element space used for spatial discretization.

{\bf The First-order Scheme:} For all $0\leq k\leq \ell-1$, given $\mathbf{u}_h^k, \mathbf{v}_h^k \in S_h$, find $\mathbf{u}_h^{k+1}, \mathbf{v}_h^{k+1} \in S_h$ such that for all $\xi \in S_h$ and $\zeta \in S_h$,
\begin{align}
(\mathbf{v}_h^{k+1}-\mathbf{v}_h^{k}, \xi) + \tau\left(\eta^{k+1} (\mathbf{v}_h^{k+1}, \xi) +  a(\mathbf{u}_h^{k+1},\xi) + \left(\alpha \left|\mathbf{u}_h^{k+1}\right|^2\mathbf{u}_h^{k+1} - \beta \mathbf{u}_h^{k}, \xi\right)\right) &= 0, \label{eq:1st-CS-FEM-1} \\
 (\mathbf{u}_h^{k+1}-\mathbf{u}_h^{k}, \zeta) - \tau(\mathbf{v}_h^{k+1}, \zeta) &= 0.  \label{eq:1st-CS-FEM-2}
\end{align}
The initial conditions are set as $\mathbf{u}_h^{0} =  R_h \mathbf{u}_0$ and $\mathbf{v}_h^0 \equiv \mathbf{0}$, where $R_h$ is the Ritz projection operator $R_h: H^1(\Omega,\mathbb{R}^N) \rightarrow S_h$, defined by:
$$
a\left(R_h \psi- \psi, \xi\right)=0, \quad \forall \xi \in S_h.
$$

{\bf The Second-order Scheme:} For all $0\leq k\leq \ell-1$, given $\mathbf{v}^{k}_h,\mathbf{u}^{k-1}_h,\mathbf{u}^{k}_h\in S_h$, find $\mathbf{v}^{k+1}_h,\mathbf{u}^{k+1}_h\in S_h$ such that for all $\xi \in S_h$ and $\zeta \in S_h$,
\begin{align}
(\mathbf{v}_h^{k+1} - \mathbf{v}_h^{k}, \xi) + \tau\left(\eta^{k+\frac{1}{2}}(\mathbf{v}_h^{k+\frac{1}{2}}, \xi) + a(\mathbf{u}_h^{k+\frac12},\xi) + \left(\alpha\chi(\mathbf{u}_h^{k+1}, \mathbf{u}_h^k) - \beta\widetilde{\mathbf{u}}_h^{k+\frac12}, \xi\right)\right) &= 0, \label{eq:2nd-CS-FEM-1} \\
(\mathbf{u}_h^{k+1}-\mathbf{u}_h^{k}, \zeta) - \tau(\mathbf{v}_h^{k+\frac{1}{2}}, \zeta) &= 0, \label{eq:2nd-CS-FEM-2}
\end{align}
where notations defined in \eqref{eq:u-k+1/2} are employed. The initial conditions are set as $\mathbf{u}_h^{-1} =\mathbf{u}_h^{0} =  R_h \mathbf{u}_0$ and $\mathbf{v}_h^0 \equiv \mathbf{0}$.

In each iteration, we solve a coupled nonlinear system using Newton's method. The initial guess for the $k$-th iteration in Newton's method is set as \(2\mathbf{u}^k-\mathbf{u}^{k-1}\), and the iteration terminates when the residual dropped below \(10^{-10}\).

To ensure consistency in our comparisons across different numerical schemes, the same termination criteria is applied. The iterative process stops when both the following conditions are concurrently met:
\begin{align}\label{eq:stopcds}
    \left\|\Delta_h \mathbf{u} - f(\mathbf{u})\right\|_{\infty}< \varepsilon_r \quad \mbox{and}\quad   
   \frac{\left\|\mathbf{u}_h^n - \mathbf{u}_h^{n-1}\right\|_{\infty}}{\tau} < \varepsilon_v.
   \end{align}
Here both $0<\varepsilon_r\ll 1$ and $0<\varepsilon_v\ll 1$ are small numbers.

\subsection{Ginzburg-Landau free energy}
We start with the case of scalar-valued functions by examining a minimization problem governed by the scalar Ginzburg-Landau free energy functional. The functional is expressed as follows:
\begin{align}\label{eq:Ginzburg-en}
    E(u) = \int_{\Omega} \left( \frac{1}{2} |\nabla u|^2 + \frac{1}{4\epsilon^2}(u^2 - 1)^2 \right) \mathrm{d} \mathbf{x},
\end{align}
where \( \epsilon>0 \)  represents a potentially very small parameter. 

\begin{example}\label{eg:1}
We first compare various numerical schemes for the second-order flow to compute the stationary state. The chosen domain is $\Omega = [0,1] \times [0,1]$, with the initial condition specified as $u_0(x,y) = xy(1-x)(1-y)$. The final time is set as $T=1$, and the spatial size is $h=1/64$. We set the parameter $\epsilon = 0.1$. In addition to our primary focus on the two convex-splliting schemes, we compare with another four alternative schemes: two first-order schemes, namely Forward Euler \eqref{eq:forward-euler-1}-\eqref{eq:forward-euler-2} and Backward Euler \eqref{eq:backward-euler-1}-\eqref{eq:backward-euler-2}, as well as two second-order schemes, the Semi-Implicit \eqref{eq:semi-implicit} and Crank-Nicolson method \eqref{eq:crank-nicolson-1}-\eqref{eq:crank-nicolson-2} . Detailed formulations of these alternative schemes are provided in Appendix~\ref{appendix:numer_scheme}.

For all the schemes, we employ a uniform stopping criterion \eqref{eq:stopcds} with $\varepsilon_r=10^{-3}$ and $\varepsilon_v=10^{-3}$. The maximum allowed termination time is set to $T=500$, correspondingly, the maximum number of iterations is set to $500/\tau$. Table \ref{tab:1st-order-groundstate} and Table \ref{tab:2nd-order-groundstate} summarize the performances of the first-order and second-order schemes, respectively. These tables detail the number of iterations (iter), the average number of inner iterations (iters) (i.e., the average number of iterations required to solve the nonlinear equations at each time step), the computed energy, and the consumed CPU time (cpu).

Numerical observations indicate that the Forward Euler scheme maintains stability only with very small time steps (\(\tau = 0.001\)) and fails to converge when larger steps are applied. We have not displayed this limitation in Table \ref{tab:1st-order-groundstate}. In contrast, both the Backward Euler and the first-order convex-splitting schemes exhibit robust stability over a diverse range of time step sizes. Notably, the first-order convex splitting scheme shows its advantages over the backward Euler scheme. As evidenced in Table \ref{tab:1st-order-groundstate}, for time step sizes $\tau=1$ and $\tau=10$, although the backward Euler scheme reaches the stopping criteria, it converges to a solution with a higher energy. This discrepancy arises because the Backward Euler method does not guarantee the decay of pseudo-energy, as depicted in Figure~\ref{fig:1-1}. Conversely, the first-order convex-splitting method consistently ensures the decay of pseudo-energy (refer to Figure~\ref{fig:1-2}), aligning with the property presented in \eqref{eq:pseudo_energy-calculate}. 

Furthermore, Table~\ref{tab:2nd-order-groundstate} underscores the superior performance of the second-order convex splitting scheme when compared to both the Crank-Nicolson and Semi-Implicit schemes. This superiority stems from the unconditional unique solvability and unconditional energy stability inherent in the convex splitting schemes—a contrast to the Crank-Nicolson scheme, which does not assure unconditional unique solvability, and the Semi-Implicit scheme, which becomes unstable with larger time steps.

These findings showcase the distinctive benefits of convex-splitting schemes, particularly their simultaneous achievement of unconditional unique solvability and unconditional energy stability, highlighting their significance over other numerical methods.

\begin{table}[!t]
\centering
\caption{Results for First-order schemes: Example~\ref{eg:1}.}
\begin{tabular}{|c|c|c|c|c|c|}
\hline
Scheme & \(\tau\) & Iter(iters) & Energy & Maxres & CPU(s) \\
\hline
Backward Euler & 10 & 3 (3.33) & 25.0000 & \(1.41 \times 10^{-7}\) & 0.40 \\
Convex-splitting 1st & 10 & 12 (4.22) & 15.6982 & \(2.39 \times 10^{-5}\) & 1.02 \\
\hline
Backward Euler & 1 & 18 (3.17) & 24.9979 & \(3.13 \times 10^{-7}\) & 1.49 \\
Convex-splitting 1st & 1 & 13 (3.62) & 15.6982 & \(1.60 \times 10^{-5}\) & 1.82 \\
\hline
Backward Euler & 0.1 & 28 (2.82) & 15.6982 & \(2.19 \times 10^{-6}\) & 2.24 \\
Convex-splitting 1st & 0.1 & 32 (2.81) & 15.6982 & \(1.48 \times 10^{-6}\) & 2.44 \\
\hline
Backward Euler & 0.01 & 251 (1.99) & 15.6983 & \(1.94 \times 10^{-5}\) & 17.52 \\
Convex-splitting 1st & 0.01 & 257 (1.99) & 15.6982 & \(1.12 \times 10^{-5}\) & 16.66 \\
\hline
Forward Euler & 0.001 & 2763 & 15.6982 & \(3.23 \times 10^{-6}\) & 85.32 \\
Backward Euler & 0.001 & 2766 (1.61) & 15.6982 & \(3.17 \times 10^{-6}\) & 155.40 \\
Convex-splitting 1st & 0.001 & 2768 (1.61) & 15.6982 & \(2.83 \times 10^{-6}\) & 153.69 \\
\hline
\end{tabular}
\label{tab:1st-order-groundstate}
\end{table}

\begin{table}[!t]
\centering
\caption{Results for Second-order schemes: Example~\ref{eg:1}.}
\begin{tabular}{|c|c|c|c|c|c|}
\hline
Scheme & \(\tau\) & Iter(iters) & Energy & Maxres & CPU(s) \\
\hline
Convex-splitting 2nd & 10 & 18 (4.67) & 15.6982 & \(3.63 \times 10^{-5}\) & 1.56 \\
Crank-Nicolson & 10 & 50 (7.56) & 15.7947 & \(1.95 \times 10^{-3}\) & 7.22 \\
\hline
Convex-splitting 2nd & 1 & 36 (4.00) & 15.6982 & \(2.63 \times 10^{-5}\) & 3.30 \\
Crank-Nicolson & 1 & 500 (4.96) & 15.7031 & \(8.16 \times 10^{-4}\) & 54.76 \\
\hline
Semi-Implicit & 0.1 & 5000 & 15.6982 & \(3.53 \times 10^{-4}\) & 166.43 \\
Convex-splitting 2nd & 0.1 & 37 (3.16) & 15.6982 & \(4.34 \times 10^{-6}\) & 3.86 \\
Crank-Nicolson & 0.1 & 32 (3.31) & 15.6982 & \(7.27 \times 10^{-6}\) & 2.90 \\
\hline
Semi-Implicit & 0.01 & 277 & 15.6982 & \(3.78 \times 10^{-6}\) & 8.95 \\
Convex-splitting 2nd & 0.01 & 277 (2.13) & 15.6982 & \(3.83 \times 10^{-6}\) & 16.69 \\
Crank-Nicolson & 0.01 & 277 (2.16) & 15.6982 & \(3.69 \times 10^{-6}\) & 16.99 \\
\hline
Semi-Implicit & 0.001 & 2767 & 15.6982 & \(3.49 \times 10^{-6}\) & 91.46 \\
Convex-splitting 2nd & 0.001 & 2767 (2.00) & 15.6982 & \(3.49 \times 10^{-6}\) & 166.78 \\
Crank-Nicolson & 0.001 & 2767 (2.00) & 15.6982 & \(3.48 \times 10^{-6}\) & 144.14 \\
\hline
\end{tabular}
\label{tab:2nd-order-groundstate}
\end{table}
\end{example}

\begin{figure}[!t]
\centering
\begin{subfigure}[b]{0.44\textwidth}  
\includegraphics[width=0.9\textwidth, height=0.18\textheight]{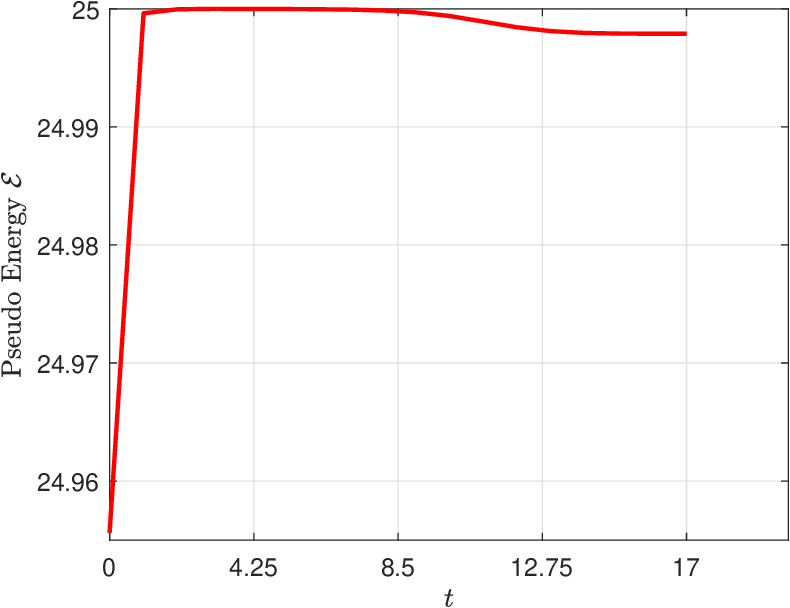}
\caption{Backward-Euler scheme}
    \label{fig:1-1}
\end{subfigure}%
\hfill
\begin{subfigure}[b]{0.44\textwidth}
\includegraphics[width=0.9\textwidth, height=0.18\textheight]{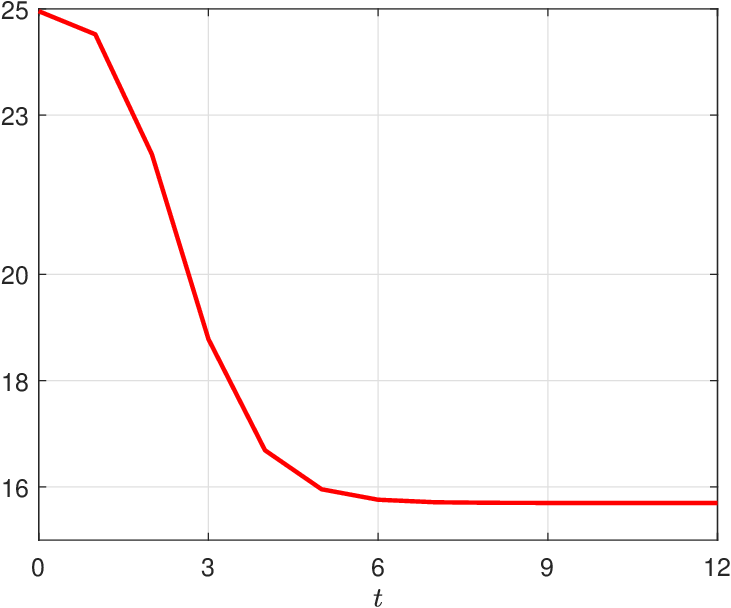}
\caption{First-order Convex-splitting scheme}
    \label{fig:1-2}
\end{subfigure}%
\caption{Pseudo-energy decay comparison: Backward Euler scheme \eqref{eq:backward-euler-1}-\eqref{eq:backward-euler-2} vs. First-order Convex-splitting scheme \eqref{eq:1st-CS-FEM-1}-\eqref{eq:1st-CS-FEM-2} for $\tau=1$ in Example~\ref{eg:1}.}
\label{fig:Pseudo-En-decay}
\end{figure}

In the following two examples of this subsection, we compare the numerical efficiency of the proposed second-order flow methods with those of the gradient flow methods. The gradient flow and its corresponding convex-splitting schemes are provided in Appendix~\ref{appendix:gradient-flow}.
 
\begin{example}\label{eg:3}
Considering a computational domain of $[0,2\pi] \times [0,2\pi]$ with the initial condition \( u_0(x,y) = \tanh(x-2y)\sin(x)\cos(x) \). We set the parameter \(\epsilon=0.05\) and use the stopping criteria \eqref{eq:stopcds} with both \(\varepsilon_r\) and \(\varepsilon_v\) fixed at \(10^{-3}\). The experiment employs four computational strategies: the first-order and second-order convex-splitting schemes for gradient flow (GF-CS-1st and GF-CS-2nd) and for second-order flow (SF-CS-1st and SF-CS-2nd). The time step size is \(\tau = 0.1\) for second-order schemes and \(\tau = 0.01\) for first-order schemes, with a spatial grid size of \(h = 1/128\).

Figure~\ref{fig:isotropic_comparison} shows the energy evolution over time for these methods. The results indicate that the second-order flow methods (SF-CS-1st and SF-CS-2nd) demonstrate a faster rate of energy decay compared to the gradient flow methods (GF-CS-1st and GF-CS-2nd), although the gradient flow methods are already quite efficient in this scenario.  
\end{example}

\begin{figure}[!t]
    \centering
    \begin{subfigure}[b]{0.4\textwidth}
        \includegraphics[width=\textwidth]{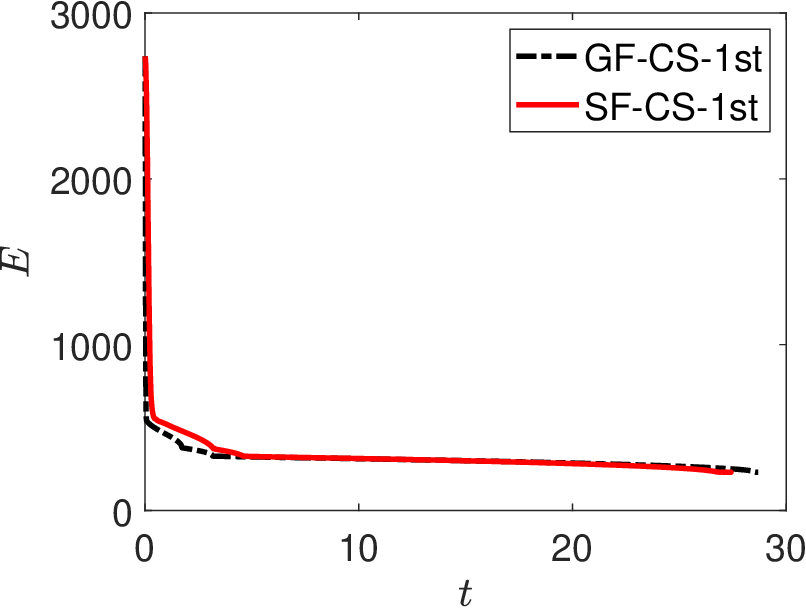}
    \end{subfigure}
    \hspace{3em}
    \begin{subfigure}[b]{0.4\textwidth}
        \includegraphics[width=\textwidth]{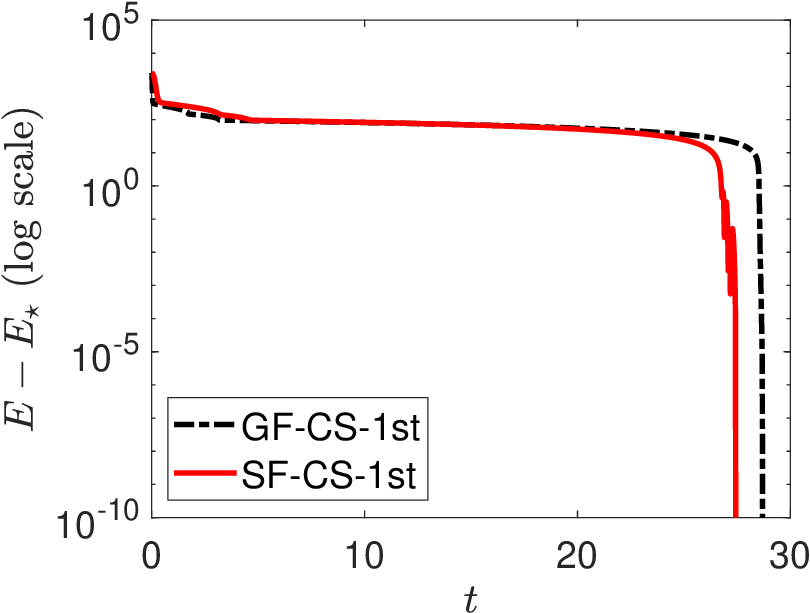}
    \end{subfigure}
    
    \vspace{1ex} 
    
    \begin{subfigure}[b]{0.4\textwidth}
        \includegraphics[width=\textwidth]{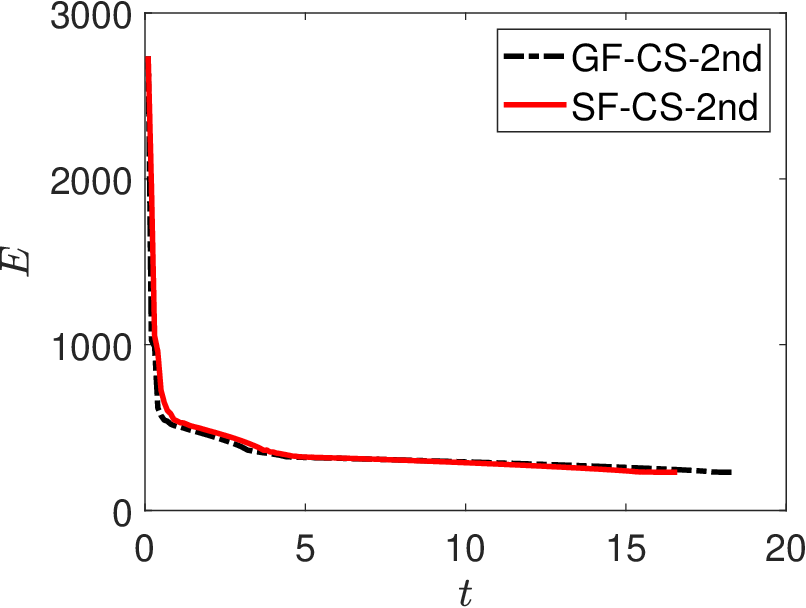}
    \end{subfigure}
    \hspace{3em}
    \begin{subfigure}[b]{0.4\textwidth}
        \includegraphics[width=\textwidth]{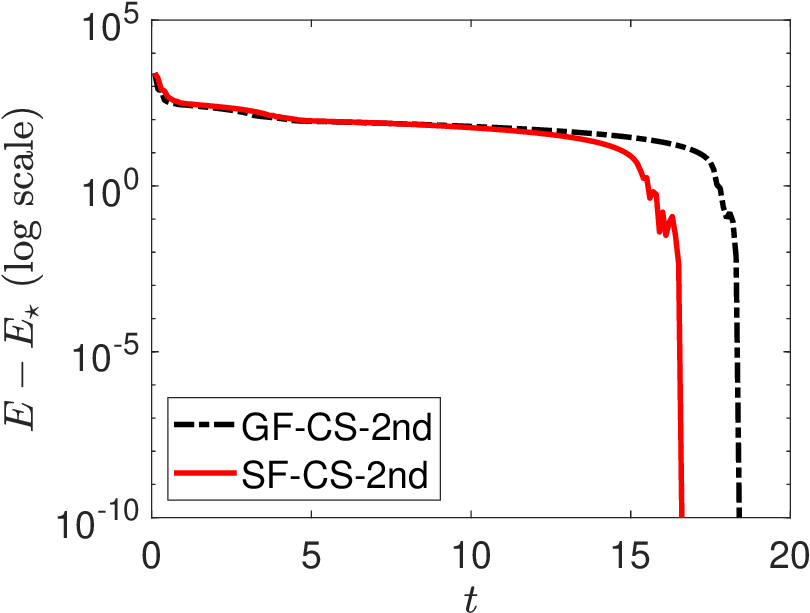}
    \end{subfigure}
    
    \caption{Comparison of energy evolution (left column) and convergence to the minimum energy (right column) over iterations for Example~\ref{eg:3} (an isotropic case) using Gradient Flow (GF-CS-1st $\&$ GF-CS-2nd) and Second-order Flow (SF-CS-1st $\&$ SF-CS-2nd) methodologies. Within the figures, $E$ denotes the calculated energy, while $E_{\star}$ signifies the minimum energy achieved throughout the iterations.}
    \label{fig:isotropic_comparison}
    \vspace{-2ex}
\end{figure}

\begin{example}\label{eg:4}
In this example, we investigate an anisotropic variant of the original Ginzburg-Landau energy \eqref{eq:Ginzburg-en}. The anisotropic energy functional, \( E_{aniso}(u) \), is defined by:
\begin{align}
E_{aniso}(u) = \int_{\Omega} \left( \frac{1}{2} \left( k_x(\partial_x u)^2 + k_y(\partial_y u)^2 \right) + \frac{1}{4\epsilon^2}(u^2-1)^2 \right) dx.
\end{align}
Here, we set \( k_x = \frac{1}{100} \) and \( k_y = 1 \) to introduce clear anisotropy, while all other parameters remain consistent with those used in Example~\ref{eg:3}. The experiment compares the rate of energy decay between the gradient flow and second-order flow. Given the lower computational efficiency and higher computational cost of the first-order schemes observed in example~\ref{eg:3}, in this more complex scenario, we only compared the results of the second-order schemes and did not evaluate the first-order schemes.

As depicted in Figure~\ref{fig:anisotropic_comparison}, the trajectories clearly indicate the superior performance of a second-order flow method over its gradient flow counterpart. The energy reduction pace is notably faster than the gradient flow method, underlining the enhanced efficiency of the second-order flow, especially in an anisotropic context.

\begin{figure}[!t]
\centering
\begin{subfigure}[b]{0.4\textwidth}
    \includegraphics[width=\textwidth]{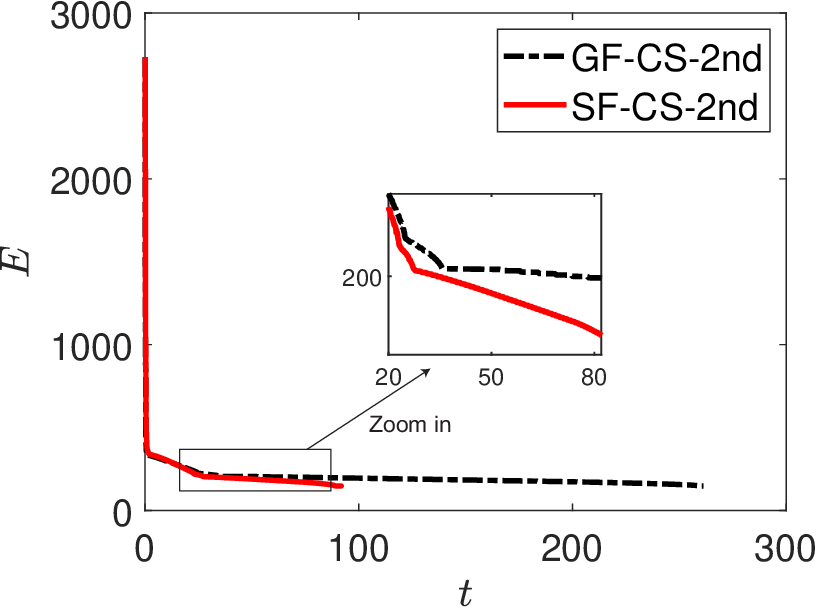}
    \label{fig1:13-1}
\end{subfigure}%
\hspace{3em}
\begin{subfigure}[b]{0.4\textwidth}
    \includegraphics[width=\textwidth]{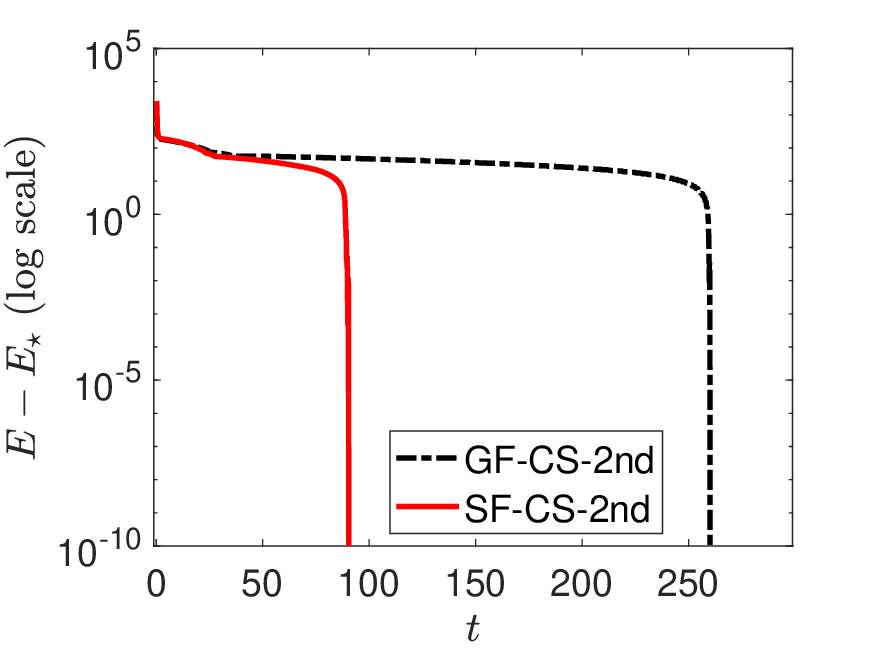}
    \label{fig1:13-2}
\end{subfigure}%
\vspace{-2ex}
\caption{Comparison of energy evolution (left column) and convergence to the minimum energy (right column) over iterations for Example~\ref{eg:4} (an anisotropic case) using Gradient Flow (GF-CS-2nd) and Second-order Flow (SF-CS-2nd) methodologies.}
\label{fig:anisotropic_comparison}
\vspace{-3ex}
\end{figure}
\end{example}

\begin{remark}
This anisotropic term reduces the uniform convexity of the quadratic term in the energy functional, which makes gradient flow methods not so efficient as in the isotropic case. This example highlights the advantages of the second-order flows, which accelerate the minimization algorithm when the convexity of the energy functional is challenged. Notice that the theoretical results established in this paper are also valid for the anisotropic form energy functional.
\end{remark}

\subsection{Landau–de Gennes (LdG) model}
The LdG model is widely recognized for its role in characterizing liquid crystal ordering, treating the free energy as a functional of the $\mathbf{Q}$-tensor field, as detailed in various studies \cite{CaiSheXu17, huang2019regularity, wang2019order, yin2020construction}. For two-dimensional nematic liquid crystals, the LdG free energy functional is given by:
\[
J(\mathbf{Q}) = \int \left\{ \frac{1}{2}|\nabla \mathbf{Q}|^2 + \vartheta F(\mathbf{Q}) \right\} \, \mathrm{d} \mathbf{r},
\]
where $\vartheta$ is a positive parameter, and the one-constant approximation \cite{Bal17} is assumed. The $\mathbf{Q}$-tensor at a point $\mathbf{r} = (x, y)$ represents the orientational order of the liquid crystals and is characterized by a $2 \times 2$ traceless and symmetric matrix:
$$
\mathbf{Q} = \begin{pmatrix}
p_1 & p_2 \\
p_2 & -p_1
\end{pmatrix}.
$$
The square of the Frobenius norm of $\nabla \mathbf{Q}$ is computed as $|\nabla \mathbf{Q}|^2 = 2\left(|\nabla p_1|^2 + |\nabla p_2|^2\right)$. The bulk free energy, or Landau function $F(\mathbf{Q})$, is most often taken to be of the general form \cite{MotNew14}:
$$
F(\mathbf{Q}) = \frac{a}{4} \operatorname{tr}\left(\mathbf{Q}^2\right) - \frac{b}{6} \operatorname{tr}\left(\mathbf{Q}^3\right) + \frac{c}{8}\left(\operatorname{tr}\left(\mathbf{Q}^2\right)\right)^2,
$$
in which $a$ signifies the reduced temperature difference, $b \geq 0$ and $c>0$. For our experiments, we consider the particular case
$$
F(\mathbf{Q}) = \frac{a}{4} \operatorname{tr}\left(\mathbf{Q}^2\right) + \frac{1}{8}\left(\operatorname{tr}\left(\mathbf{Q}^2\right)\right)^2 = \frac{a}{2}(|p_1|^2 + |p_2|^2) + \frac{1}{2}(|p_1|^2 + |p_2|^2)^2.
$$
We then introduce $\mathbf{u} = ( p_1, p_2)^T$, the energy functional becomes
\begin{align}
E(\mathbf{u}) &= \int_{\Omega}\left\{|\nabla \mathbf{u}|^2 + \vartheta \left( \frac{a}{2}|\mathbf{u}|^2 + \frac{1}{2}|\mathbf{u}|^4 \right)\right\} \mathrm{d} \mathbf{r}.
\end{align}
It follows that $J(\mathbf{Q})$ is equivalent to $E(\mathbf{u})$. The objective is to find stable liquid crystal configurations by minimizing the energy functional $E(\mathbf{u})$, subject to certain boundary conditions. This problem aligns with the energy minimization problem \eqref{eq:energy_minimization}. After a minimizer $\mathbf{u}$ of $E$ is found, the generally nonuniform, undiagonalized $\mathbf{Q}(\mathbf{r})$ can be analyzed by computing its eigenvalues,
$$
S(\mathbf{r}) / 2= \pm \sqrt{p_1(\mathbf{r})^2+p_2(\mathbf{r})^2} .
$$
The nematic field director is given by the eigenvector of $\mathbf{Q}(\mathbf{r})$ (for example, corresponding to the positive eigenvalue),
$$
\mathbf{n}(\mathbf{r})=\left(\sqrt{1 / 2+p_1(\mathbf{r}) /|S(\mathbf{r})|}, \sigma(\mathbf{r}) \sqrt{1 / 2-p_1(\mathbf{r}) /|S(\mathbf{r})|}\right),
$$
with $\sigma(\mathbf{r})=\mathbf{1}_{\left\{p_2 \geq 0\right\}}(\mathbf{r})-\mathbf{1}_{\left\{p_2<0\right\}}(\mathbf{r})$ and $\mathbf{1}$ is the indicator function.

\begin{example}
To ensure the system in the nematic phase, we select a low temperature setting with $a = -1.672$. The domain $\Omega = [-1,1] \times [-1,1]$ is selected for square confinement. For the Dirichlet boundary conditions, we impose
$$
\mathbf{u}(x, y = \pm 1) = \frac{S_0}{2}\begin{pmatrix}
1 & 0 \\
0 & -1
\end{pmatrix}, \quad \mathbf{u}(x = \pm 1, y) = \frac{S_0}{2}\begin{pmatrix}
-1 & 0 \\
0 & 1
\end{pmatrix},
$$
where $S_0=\sqrt{2|a|}$. The chosen boundary conditions for the Q-tensor, aligning the nematic director parallel to the square boundary lines, represents a common approach in the modeling of liquid crystals confined within geometries \cite{deGennes1993}. To compute the stationary points of the LdG free energy, the second-order flow method is applied, utilizing the convex-splitting scheme \eqref{eq:2nd-CS-FEM-1}-\eqref{eq:2nd-CS-FEM-2}. We select a time step size of $\tau=0.1$, initiate our simulations with random initial conditions, and specify the tolerances for the stopping criteria in \eqref{eq:stopcds} as $\varepsilon_r=10^{-10}$ and $\varepsilon_v=10^{-10}$. The results depict stable liquid crystal configurations for varying values of $\vartheta$. Specifically, for $\vartheta$ values of 5, 15, and 50, we use a grid size of $h=1/64$, while for $\vartheta$ values of 100, 200, and 500, we opt for a finer grid size of $h=1/128$. The computed configurations are presented in Figure~\ref{fig:Liquid-crystal}.
\end{example}

\begin{figure}[!ht]
\vspace{-1ex}
    \centering
    \begin{subfigure}{0.31\textwidth}
        \includegraphics[width=\linewidth]{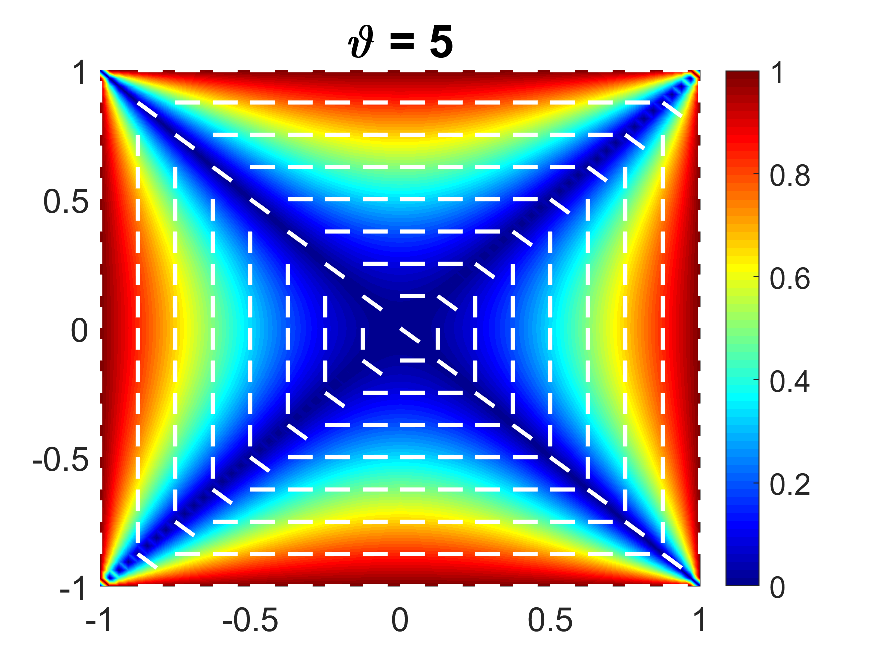}
    \end{subfigure}
    \begin{subfigure}{0.31\textwidth}
        \includegraphics[width=\linewidth]{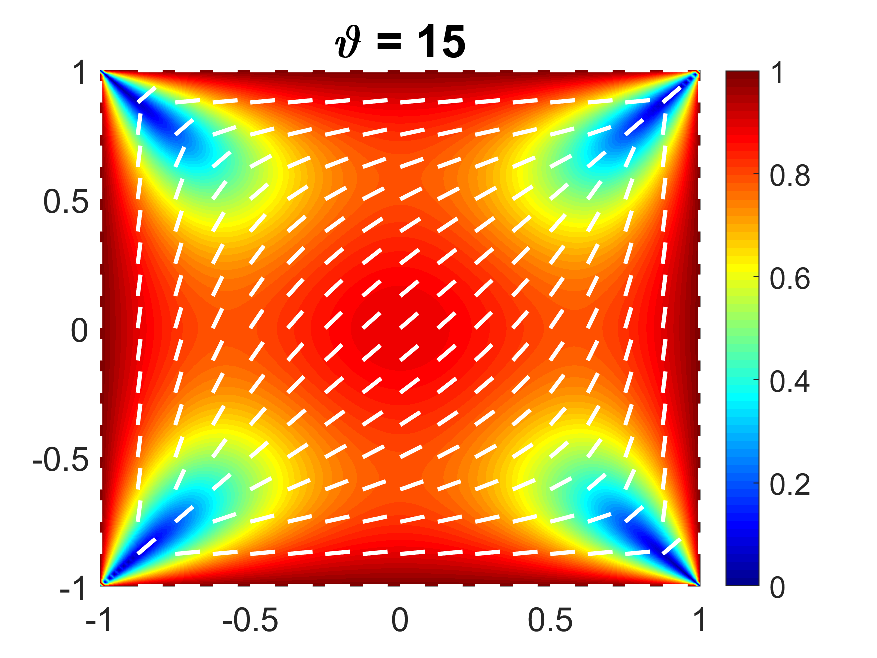}
    \end{subfigure}
    \begin{subfigure}{0.31\textwidth}
        \includegraphics[width=\linewidth]{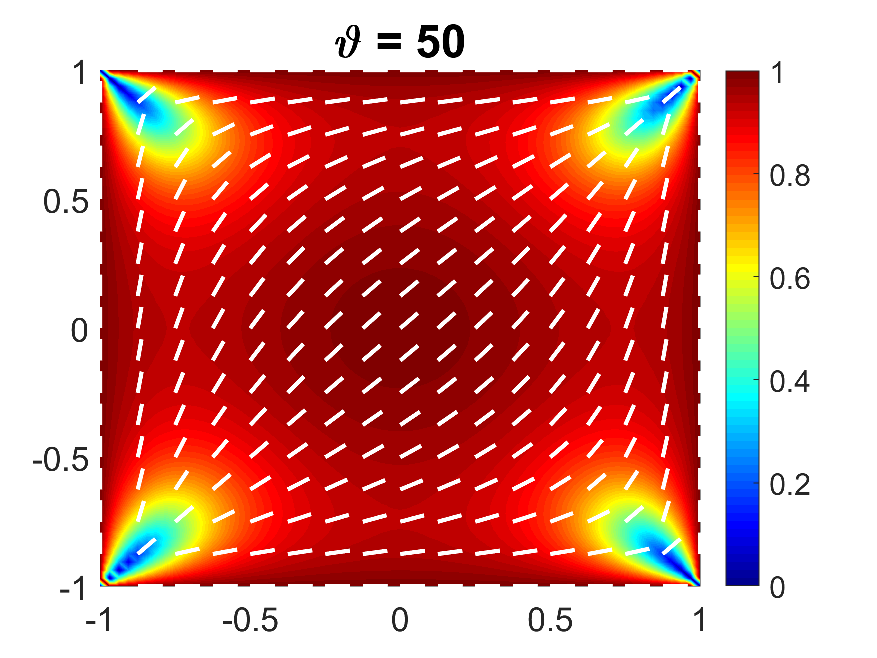}
    \end{subfigure}
    
    
    \begin{subfigure}{0.31\textwidth}
        \includegraphics[width=\linewidth]{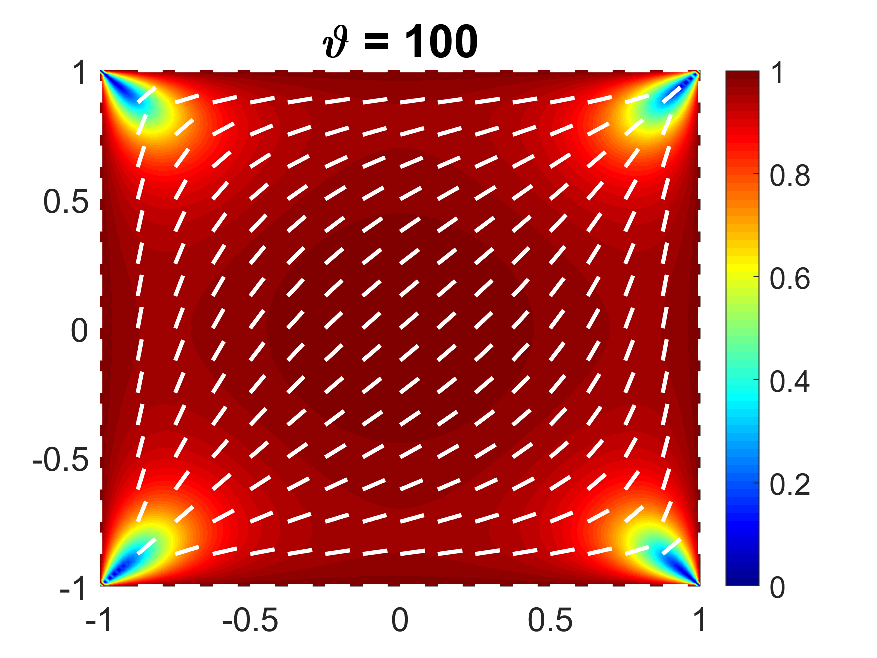}
    \end{subfigure}
    \begin{subfigure}{0.31\textwidth}
        \includegraphics[width=\linewidth]{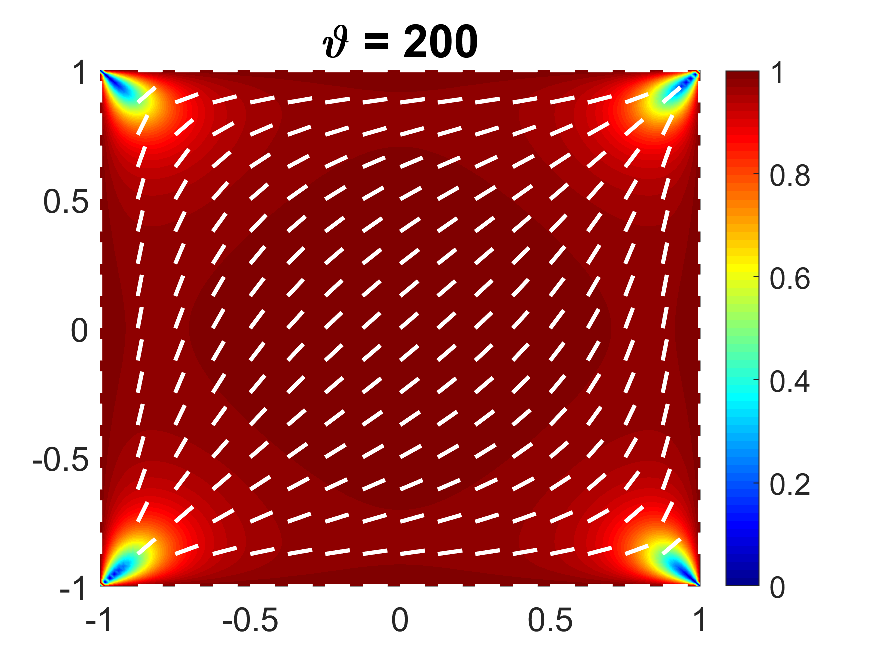}
    \end{subfigure}
    \begin{subfigure}{0.31\textwidth}
        \includegraphics[width=\linewidth]{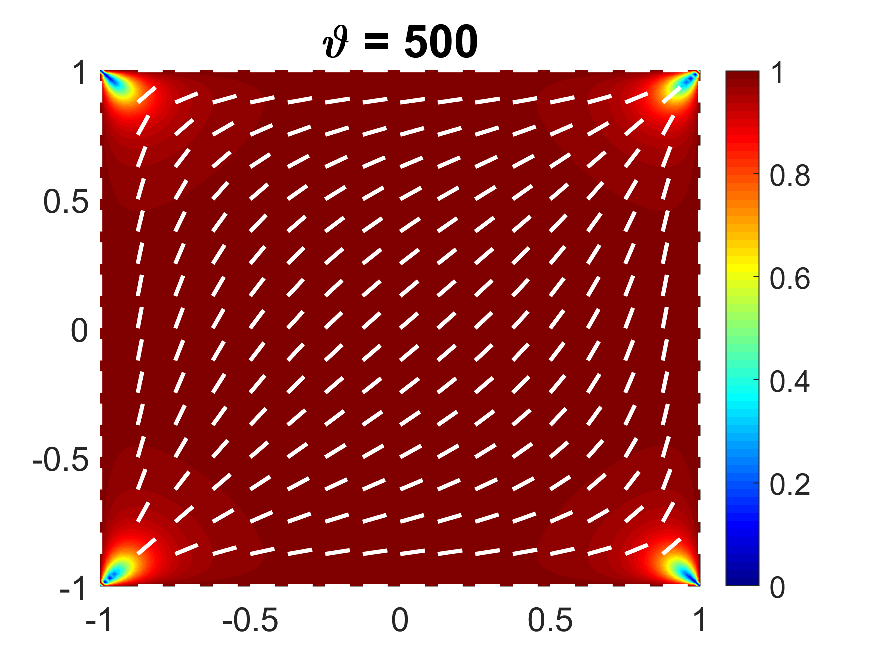}
    \end{subfigure}
    \caption{Stable liquid crystal configurations at varying $\vartheta$ values (5, 15, 50, 100, 200, 500), computed via the second-order flow method. The color gradient indicates the relative intensity of directional ordering, quantified as $|S(\mathbf{r})| / S_0$. The white bars depict the directional vectors of the nematic field, $\mathbf{n}(\mathbf{r})$.}
    \label{fig:Liquid-crystal}
    \vspace{-3ex}
\end{figure}

\section{Conclusion}\label{section:5}
This paper studied  numerical solvers based on second-order flows for a class of non-convex variational problems. Starting with semi-discrete formulations of the second-order flows, we have proven subsequential convergence of the temporal trajectory to a stationary point of the non-convex energy, and established the well-posedness of the continuous flows by making the timestep vanish. For a class of non-convex functionals, both the schemes exhibit unconditional pseudo-energy stability and unconditional unique solvability simultaneously, highlighting their advantages over other numerical schemes. Moreover, we have witnessed the efficiency of second-order flows, as the minimization strategy for these non-convex variational problems, especially notable when dealing with functional of increased complexity and non-convexity. Their versatility has also been affirmed through a couple of applications in scientific computing. This investigation paves the way to deepen our exploration of second-order flows. Interesting aspects from the optimization point of view, like the convergence of complete trajectories and associated convergence rates to stationary points, are on the agenda of our next steps. Moreover, alternative numerical schemes, such as stabilized IMEX BDF methods or exponential time differential methods, will be interesting potential topics, in particular their numerical analysis and energy stability analysis.

\section*{Acknowledgments}
\addcontentsline{toc}{section}{Acknowledgments}
The authors thank the reviewers for their comments which helped to improve the presentation of the paper. Part of the work was done when W.~Liu was employed by the National University of Singapore.

\appendix

\section{Other numerical methods involved in the paper}
\subsection{Other schemes for discretizing second-order flows}\label{appendix:numer_scheme}

\paragraph{Forward Euler Scheme}
\begin{align}
& (\mathbf{v}_h^{k+1}-\mathbf{v}_h^k, \xi) + \tau \eta^{k} (\mathbf{v}_h^{k}, \xi) + \tau \left( a(\mathbf{u}_h^{k},\xi) + (f(\mathbf{u}_h^{k}) , \xi) \right) = 0, \quad \forall \xi \in S_h, \label{eq:forward-euler-1} \\
& (\mathbf{u}_h^{k+1}-\mathbf{u}_h^k, \zeta) - \tau (\mathbf{v}_h^{k+1}, \zeta) = 0, \quad \forall \zeta \in S_h, \label{eq:forward-euler-2}
\end{align}
with initial conditions \(\mathbf{u}_h^{0} = R_h \mathbf{u}_0\) and \(\mathbf{v}_h^0 \equiv \mathbf{0}\).

\paragraph{Backward Euler Scheme}
\begin{align}
& (\mathbf{v}_h^{k+1}-\mathbf{v}_h^k, \xi) + \tau \eta^{k+1} (\mathbf{v}_h^{k+1}, \xi) + \tau \left( a(\mathbf{u}_h^{k+1},\xi) + (f(\mathbf{u}_h^{k+1}) , \xi) \right) = 0, \quad \forall \xi \in S_h, \label{eq:backward-euler-1} \\
& (\mathbf{u}_h^{k+1}-\mathbf{u}_h^k, \zeta) - \tau (\mathbf{v}_h^{k+1}, \zeta) = 0, \quad \forall \zeta \in S_h, \label{eq:backward-euler-2}
\end{align}
with initial conditions \(\mathbf{u}_h^{0} = R_h \mathbf{u}_0\) and \(\mathbf{v}_h^0 \equiv \mathbf{0}\).

\paragraph{Semi-implicit Scheme}
\begin{align}
    (\mathbf{u}_h^{k+1} - 2 \mathbf{u}_h^k + \mathbf{u}_h^{k-1}, \xi) &+ \frac{\tau\eta^{k+\frac{1}{2}}}{2}(\mathbf{u}_h^{k+1} - \mathbf{u}_h^{k-1}, \xi) \nonumber \\
    &+ \tau^2\left( a\left(\frac{\mathbf{u}_h^{k+1} + \mathbf{u}_h^{k-1}}{2}, \xi\right) + (f(\mathbf{u}_h^{k}), \xi) \right) = 0, \quad \forall \xi \in S_h, \label{eq:semi-implicit}
\end{align}
with the initial conditions \(\mathbf{u}_h^{0} = R_h \mathbf{u}_0\) and \(\mathbf{v}_h^0 \equiv \mathbf{0}\). The first step is given by
\begin{equation}
    (\mathbf{u}_h^{1}, \xi) = (\mathbf{u}_h^{0}, \xi) - \frac{\tau^2}{2}\left( a(\mathbf{u}_h^{0}, \xi) + (f(\mathbf{u}_h^{0}), \xi) \right), \quad \forall \xi \in S_h.
\end{equation}      
\paragraph{Crank-Nicolson Scheme}
\begin{align}
& (\mathbf{v}_h^{k+1}-\mathbf{v}_h^k, \xi) + \tau \eta^{k+\frac{1}{2}} (\mathbf{v}_h^{k+\frac{1}{2}}, \xi) \nonumber \\
& \qquad + \tau \left( a(\mathbf{u}_h^{k+\frac{1}{2}},\xi) + \left(\frac{F(\mathbf{u}_h^{k+1})-F(\mathbf{u}_h^k)}{\mathbf{u}_h^{k+1}-\mathbf{u}_h^k}, \xi\right) \right) = 0, \quad \forall \xi \in S_h, \label{eq:crank-nicolson-1} \\
& (\mathbf{u}_h^{k+1}-\mathbf{u}_h^k, \zeta) - \tau (\mathbf{v}_h^{k+\frac{1}{2}}, \zeta) = 0, \quad \forall \zeta \in S_h, \label{eq:crank-nicolson-2}
\end{align}
with initial conditions \(\mathbf{u}_h^{-1}=\mathbf{u}_h^{0} = R_h \mathbf{u}_0\) and \(\mathbf{v}_h^0 \equiv \mathbf{0}\).

\subsection{Gradient flow and its convex-splitting schemes}\label{appendix:gradient-flow}
We consider the following equation
\begin{align}
\begin{cases}\label{eq:gradient_flow}
\partial_{t}{\mathbf{u}}=\Delta \mathbf{u}-f(\mathbf{u}), & \text { in } [0, T] \times \Omega, \\
  \mathbf{u}(0)=\mathbf{u}_0,  & \text { in } \Omega, \\
   \mathbf{u}=0, & \text { on }[0, T] \times \partial \Omega.
\end{cases}.
\end{align}
Likewise, for the gradient flow \eqref{eq:gradient_flow}, we introduce a first-order and a second-order fully discrete convex-splitting scheme respectively as follows:

\paragraph{First-order convex-splitting scheme for gradient flow}\mbox{}\\
\begin{align}
 (\mathbf{u}_h^{k+1}-\mathbf{u}_h^k, \xi)+ \tau \left( a(\mathbf{u}_h^{k+1},\xi) + (f_c(\mathbf{u}_h^{k+1}) - f_e(\mathbf{u}_h^k), \xi) \right) = 0 \quad \forall \xi \in S_h, \label{eq:1st-CS-GF}
\end{align}
Here, $\tau > 0$ is the time step size, and the initial conditions are given by $\mathbf{u}_h^{0} = R_h \mathbf{u}_0$.
\vspace{2mm}

\paragraph{Second-order convex-splitting scheme for gradient flow}\mbox{}\\
\begin{align}
(\mathbf{u}_h^{k+1}-\mathbf{u}_h^k, \xi)  + \tau\left(a(\mathbf{u}_h^{k+\frac12},\xi) + \left(\chi(\mathbf{u}_h^{k+1}, \mathbf{u}_h^k) - f_e(\widetilde{\mathbf{u}}_h^{k-\frac12}), \xi\right)\right) = 0 \quad \forall \xi \in S_h, \label{eq:2nd-CS-GF} 
\end{align}
The initial condition is setting as $\mathbf{u}_h^{-1} =\mathbf{u}_h^{0} =  R_h \mathbf{u}_0$.

Gradient flow algorithms will check the residual of the Euler-Lagrange equation as the termination criteria.

\bibliographystyle{siamplain}
\bibliography{references}
\end{document}